\documentclass{amsart}
\usepackage{pstricks,pst-node}
\usepackage{amsmath,amssymb,multicol}
\usepackage{hyperref}

\newtheorem{Theo}{Theorem}
\newtheorem{Lem}[Theo]{Lemma}
\newtheorem{Cor}[Theo]{Corollary}

\newcommand{\N}{\mathbb{N}}
\newcommand{\Z}{\mathbb{Z}}

\newcommand{\ii}{\imath}

\begin{document}
\baselineskip=17pt
\thanks{We would like to thank the Rechenzentrum Universit\"at Freiburg
for providing computing resources.}

\author[G. Bhowmik]{Gautami Bhowmik}
\address{%
Gautami Bhowmik \\
Universit\'e de Lille 1\\
Laboratoire Paul Painlev\'e\\
U.M.R. CNRS 8524\\
59655 Villeneuve d'Ascq Cedex\\
France  \\
bhowmik@math.univ-lille1.fr
}
\author[I. Halupczok]{Immanuel Halupczok}
\thanks{The second author was supported by the Agence National de la Recherche
(contract ANR-06-BLAN-0183-01) and by the Fondation Sciences math\'ematiques
de Paris.}
\address{%
 Immanuel Halupczok\\
 DMA de l'ENS \\
 UMR 8553 du CNRS \\
 45, rue d'Ulm  \\
 75230 Paris Cedex 05 \\
 France\\
 math@karimmi.de
}

\author[J.-C. Schlage-Puchta]{Jan-Christoph Schlage-Puchta}
\address{%
 Jan-Christoph Schlage-Puchta\\
 Albert-Ludwigs-Universit\"at\\
 Mathematisches Institut\\
 Eckerstr. 1 \\
 79104 Freiburg\\
 Germany \\
 jcp@math.uni-freiburg.de
}

\title[Zero-sum free Sequences]{The structure of maximal zero-sum
  free Sequences}
\maketitle
MSC-Index: 11B50, 11B75, 05D05

\section{Introduction and Results}
This paper is a continuation of our investigation of zero-sum (free)
sequences of finite abelian groups (see \cite{Montreal} or \cite{trio1}). 
As is the tradition, we
let $G$ be a finite abelian group, $A\subseteq G$ a multiset and we say
that $A$ is zero-sum free if there exists no non-empty subset $B\subseteq A$,
such that $\sum_{b\in B} b=0$. Obviously, in a fixed group $G$ a
zero-sum free subset cannot be arbitrarily large. The least integer $n$
such that there does not exist a zero-sum free set with $n$ elements is
usually called the Davenport's constant of $G$, for which we write
$D(G)$. For an overview of this and related problems as well as
applications see \cite{GHK}. 

Here we consider groups of the form $\Z_n^2$, where
$\Z_n=\Z/n\Z$. Mann and Olson \cite{MO} and Kruswijk \cite{Baa} showed that 
$D(\Z_n^2)=2n-1$. Knowing the precise structure of all
counterexamples, i.e. zero-sum free sets of  $2n-2$ elements 
would simplify some inductive arguments for groups of rank $\geq 3$,
where the Davenport constant is unknown. Up to an automorphism of the
group all known examples of zero-sum free sets of maximal size are one
of the following: Either $(1, 0)$ occurs with multiplicity
$n-1$, and all other points are of the form $(a_i, 1)$, or $(1, 0)$
occurs with multiplicity $n-2$, all other points are of the form
$(a_i, 1)$, and we have $\sum_{i=1}^n a_i=1$. We are thus motivated to study
the following property introduced by Gao and Geroldinger \cite{GaoB}
 
Let $n$ be an integer. Then $n$ is said to satisfy {\em property B},
or $B(n)$ holds true, if in every maximal zero-sum free subset of
$\Z_n^2$ some element occurs with multiplicity at least $n-2$.
It is easy to see that this definition is equivalent to the statement
that every zero-sum free set of $2n-2$ elements is of one of the two
forms cited above.  

Gao and Geroldinger \cite{GaoB} proved that $B(n)$ holds true for
$n\leq 7$, and that for $n\geq 6$, $B(n)$ implies $B(2n)$. Recently,
Gao, Geroldinger and Grynkiewicz \cite{GGG} showed that property B is
almost multiplicative, that is, if $B(n)$ and $B(m)$ hold true, then
so does $B(nm)$, provided that $mn$ is odd and greater than $9$. Hence,
combining the results of \cite{GaoB} and \cite{GGG}
it suffices to prove $B(n)$ when $n$ is prime and when $n \in \{8,9,10\}$.

From now on, $p$ will always be a prime number.
If one tries to prove $B(p)$ by sheer force, the
most difficult cases are those which are close to the known maximal
zero-sums, that is, some point $a$ has multiplicity only slightly less
than $p-2$, and all other points occur in one coset of the subgroup
generated by $a$. Further the method of exponential sums runs into serious
problems with situations in which few points occur with high
multiplicity. Therefore, it appears worthwhile to deal with the case
of high multiplicities in a uniform way. The aim of this article is to
initiate a systematic approach to property B via the highest occurring
multiplicities. 

In one direction we have the following.

\newcommand{\refmm}[1]{Theorem~\ref{thm:MaxMult}~(\ref{mm:#1})}

\begin{Theo}
\label{thm:MaxMult}
Let $A\subseteq\Z_p^2$ be a set of cardinality $2p-2$,
and let $m_1\geq m_2\geq m_3$ be the
largest occurring multiplicities. Suppose that $m_1\leq p-3$, and that
one of the following statements is true:
\begin{enumerate}
\item\label{mm:1max} $m_1=p-3$
\item\label{mm:asymp} $p > N$ and $p - m_1 < cp$,
where $N, c > 0$ are two constants not depending on $p$
\item\label{mm:2max} $m_2\geq 2p/3$
\item\label{mm:3max} $m_1+m_2+m_3\geq 2p-5$
\end{enumerate}
Then $A$ contains a zero-sum.
\end{Theo}

Lettl and Schmid \cite{LS} proved the existence of a zero-sum under the
fourth condition with $2p-5$ replaced by $2p-2$. Our proof of the
fourth statement does not involve any new ideas. However, using the
first and the third condition we immediately obtain a good lower
bound for $m_3$ which greatly simplifies our arguments. With more
effort one can replace $2p-5$ by some other function of the form
$2p-c$, however, we do not feel that the amount of work necessary to
do so would be justified. The fourth statement appears to be rather
technical, the reason that we still believe it
to be of some interest is the fact that when one tries to tackle larger
group by an inductive argument along the lines of \cite{Montreal}, one
is automatically lead to situations where $m_1+m_2+m_3$ is close to $2p-2$.

In the opposite direction we combine exponential sums with
combinatorial methods to prove the following.

\begin{Theo}
\label{thm:MinMult}
There is a positive constant $\delta$, such that each set
$A\subseteq\Z_p^2$ with $|A|=2p-2$ and $m_2<\delta p$ contains a zero-sum.
\end{Theo}
Gao, Geroldinger and Schmid \cite[Theorem~4.1]{GGS} had already shown
the existence of a zero-sum under the assumption $m_1<p^{1/4-\epsilon}$.

We did not try to obtain a good numerical bound for $c$, a rather careless
estimate shows that $c=4\cdot 10^{-7}$ is admissible, which is
certainly far from  optimal. However, any value of $c$ less than $0.1$
would be of little help concerning the computational confirmation of
property B, nor do we expect much structural information for maximal
zero-sum free sets from such a small value, therefore we did not try
to optimise our estimate.

\begin{figure}
\begingroup
\psset{unit=2.5mm,dimen=middle,linewidth=0.1mm,hatchwidth=0.1mm}
\SpecialCoor
% comm(#1)(#2){blah}: put comment "blah" at the left of #2, pointing to #1; (#2) may be omitted
% same for commr, but comment is put at the right.
\makeatletter
\def\comm(#1){\@ifnextchar({\comm@(#1)}{\comm@(#1)(#1 exch pop -2 exch)}}
\def\commr(#1){\@ifnextchar({\comm@r(#1)}{\comm@r(#1)(#1 exch pop 14 exch)}}
\def\comm@(#1)(#2)#3{\psline(! #1)(! #2)\rput[r](! #2 exch 0.1 sub exch){#3\strut}}
\def\comm@r(#1)(#2)#3{\psline(! #1)(! #2)\rput[l](! #2 exch 0.1 add exch){#3\strut}}
\makeatother
\begin{pspicture}(-5,-2)(17,13)
{
  \psset{linewidth=0.2mm}
  \psline{->}(0,0)(14,0)\rput[l](14.5,0){{\small $m_1$}}
  \psline{->}(0,0)(0,12)\rput[b](0,12.5){{\small $m_2$}}
  \psline(0,0)(12,12)
}
{
  \psset{fillstyle=hlines,hatchsep=.6mm}
  \pspolygon(11,1)(12,1)(12,8)(11,8)
  \pspolygon(0,0)(12,0)(12,1)(1,1)
  \pspolygon(0,0)(12,0)(12,1)(1,1)
  \pspolygon(8,8)(12,8)(12,12)
}
\psline(12,1)(12.3,1)\rput[l](12.5,1){{\footnotesize $cp$}}
\psline(12,8)(12.3,8)\rput[l](12.5,8){{\footnotesize $\frac{2}{3}p$}}
\psline(12,12)(12.3,12)\rput[l](12.5,12){{\footnotesize $p$}}
\psline(11,0)(11,-.3)\rput[t](11,-.5){{\footnotesize $(1-C)p$}}

\comm(2 .5){{\scriptsize Theorem~\ref{thm:MinMult}}}
\commr(11.5 5){{\scriptsize \refmm{asymp}}}
\commr(11 9.5){{\scriptsize \refmm{2max}}}
\end{pspicture}
\endgroup
\caption{Property B is proven if $p$ is sufficiently big and
$(m_1,m_2)$ lies in the hatched area; $c$ and $C$ are two constants
not depending on $p$.}
\label{fig:m1m2}
\end{figure}
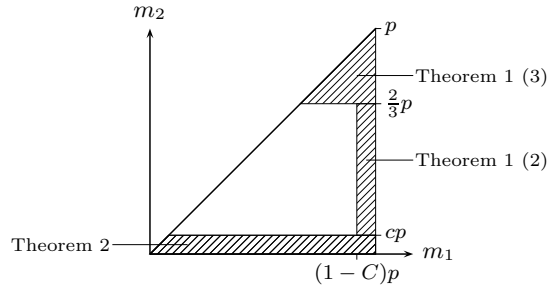

\medskip

For several of our results, the proof gets more and more complicated
as $p$ becomes small. Thus, to simplify the manual parts of the proof,
we verified as many cases as possible by brute force using a
computer. We also tried how far we could get proving property B
completely by computer. In particular, we also considered the missing non-prime
cases 8, 9 and 10. The following Theorem summarizes the results obtained this
way.
\newcommand{\refmc}[1]{Theorem~\ref{thm:MaxComp}~(\ref{mc:#1})}

\begin{Theo}
\label{thm:MaxComp}
Let $A\subseteq\Z_p^2$ be a set, and let $m_1\geq m_2\geq m_3$ be the
largest occurring multiplicities. Suppose that $m_1\leq p-3$, and that
one of the following statements is true:
\begin{enumerate}
\item\label{mc:all}
$p \le 23$
\item\label{mc:2max} $m_2\ge 2p/3$ and $m_1+m_2\leq 2p-14$
\item\label{mc:3max}
$p \le 37$ and $m_1 + m_2 + m_3 \ge 2p - 5$.
\end{enumerate}
Then $A$ contains a zero-sum. Moreover, property B holds true for 8,
9, and 10.
\end{Theo}

Part (2) and (3) do not have any merit in itself, but serve only as an
aide in the proof of Theorem~\ref{thm:MaxMult}. 

In view of the multiplicativity results of 
\cite{GaoB} and \cite{GGG}, they yield:

\begin{Cor}
Any $n \le 28$ has property B.
\end{Cor}

The remainder of this article is organized as follows.
In the next section, we list some general lemmas which we will need
later. In Sections~\ref{sect:2MaxMult} to \ref{sect:MinMult}, we prove the different statements
of Theorems~\ref{thm:MaxMult} and \ref{thm:MinMult}, approximately
in the order in which they rely upon each other. Finally, in
Section~\ref{sect:algo} we describe the algorithm used for
Theorem~\ref{thm:MaxComp}.

The following diagram describes the dependencies; $\Rnode{A}{A} \qquad \Rnode{B}{B}$
means that $A$ is used in the proof of $B$.
\psset{nodesep=.5ex}
\ncline{->}{A}{B}

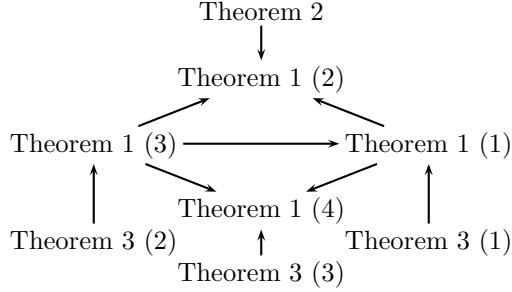
\begin{center}
\begin{tabular}{c@{}c@{}c}
&            \Rnode{MinMult}{Theorem~\ref{thm:MinMult}} & \\[3ex]
&            \Rnode{asymp}{\refmm{asymp}} &               \\[3ex]
\Rnode{2max}{\refmm{2max}} & & \Rnode{1max}{\refmm{1max}} \\[3ex]
&            \Rnode{3max}{\refmm{3max}} &                 \\
\Rnode{c2max}{\refmc{2max}} & & \Rnode{c1max}{\refmc{all}}\\
&            \Rnode{c3max}{\refmc{3max}} &
\end{tabular}
\ncline{->}{2max}{asymp}\ncline{->}{2max}{3max}\ncline{->}{2max}{1max}
\ncline{->}{1max}{3max}\ncline{->}{1max}{asymp}
\ncline{->}{c3max}{3max}\ncline{->}{c1max}{1max}
\ncline{->}{c2max}{2max}
\ncline{->}{MinMult}{asymp}
\end{center}

Note that apart from Theorem~\ref{thm:MaxComp}, there is a second place
where computer results are used: Lemma~\ref{Lem:OlsonFmc} below has been
proven using a computer, and this lemma is used in the proof of \refmm{1max}.
For $p$ sufficiently big, it can be replaced by Lemma~\ref{Lem:Olson}.
However, even for arbitrarily big $p$, \refmc{2max} is needed for
\refmm{2max}; thus apart of Theorem~\ref{thm:MinMult}, all our results
depend on the computer even for big $p$.

% However, the only place where we use that lemma
%is the proof of \refmm{1max}, where it can easily be replaced by
%Lemma~\ref{Lem:Olson} if we suppose $p$ sufficiently big. Thus for
%$p$ sufficiently big, none of our results depend on computer calculations.

\section{Auxiliary results}

$\Z_p$ is not an ordered group; however, for our purpose it is useful
to view elements such as 5 and 6 to be close together, and elements
such as 2 to be small. Of course, this notion does not make sense from
a group-theoretic point of view, since $\mathrm{Aut}(\Z_p)$ acts
transitively on $\Z_p\setminus\{0\}$. However, after fixing the generator
$1$, it makes sense to talk about the distance and the size of elements
in $\Z_p$. To be precise, we define two functions $\Z_p\rightarrow\Z$
as follows. For an element 
$a\in\Z_p$ denote by $|a| = \min\{|a'| : a' \in \Z, a' \bmod p = a\}$
the modulus of the least absolute remainder of $a$, and by $\ii(a) =
\min\{a' \ge 0 : a' \bmod p = a\}$ the least positive remainder of
$a$. When we compare elements of $\Z_p$, then we implicitly
apply $\ii$ before. For example for elements $a,b\in\Z_p$,
we write $a < b$ to mean $\ii(a) < \ii(b)$ and $a\in[x,
2x]$ to mean $\ii(a)\in[x, 2x]$. However, at some places it is
important to distinguish between $\sum_{a\in A}\ii(a)$ and
$\ii\big(\sum_{a\in A} a\big)$.

For a multiset $A$ we denote by $\Sigma(A)$ the {\em set} (not
multiset) of all subset sums of $A$, for example, $\Sigma(\{1,
1\})=\{0, 1, 2\}$, and $\Sigma_k(A)$ is the {\em set} of all subset
sums of $A$ of length $k$, for example, $\Sigma_2(\{1, 1, 2\})=\{2, 3\}$.
\begin{Lem}
\label{Lem:Onedim}
\begin{enumerate}
\item Let $A\subseteq\Z_p$ be a multiset of size $k$ without zero-sums. Then
  there are at least $k$ distinct elements representable as non-empty
  subset sums of 
  $A$, and equality holds true if and only if all elements in $A$ are
  equal.
\item Let $A\subseteq\Z_p$ be a multiset of size $p+k$ with $0\leq k\leq p-2$ without
zero-sums of length $p$. There are at least $k+1$ distinct sums of $p$
elements in $A$, and equality holds if and only if $|A|=p$ or $A$ contains
only two distinct elements.
\end{enumerate}
\end{Lem}
\begin{proof}
(1) We prove our claim by induction on $k$. For $k=1$ and $k=2$
the statement is obvious, similarly, if all elements of $A$ are equal. Now
suppose that $A$ contains at least two distinct elements, and let
$A=\{x_1, \ldots, x_k\}$ with $x_1\neq x_2$. The induction hypothesis implies
that the set $\Sigma$ of elements representable as non-empty subset sums of
$x_1,\ldots, x_{k-1}$
contains at least $k$ elements, thus, we only have to show that
$(\sum\cup\{0\})+\{0,x_k\}\neq\sum\cup\{0\}$. Suppose otherwise. Then $x_k\in\sum$, thus, the
subgroup $\langle x_k\rangle$ generated by $x_k$ is contained in $\sum\cup\{0\}$; in
particular, $-x_k\in\sum$. However, this contradicts the assumption that
$A$ does not contain a non-empty zero-sum subset.

(2) This is a result of Bollobas and Leader \cite{BoLe}.
\end{proof}
The following is probably the first non-trivial result proved on sumsets in finite
abelian groups.
\begin{Lem}[Cauchy-Davenport]
Let $A, B\subseteq\Z_p$ be sets containing no element twice. Then
$|A+B|\geq\min(|A|+|B|-1, p)$, where $A+B$ is interpreted as a set (not a
multiset).
\end{Lem}
We shall repeatedly use this theorem in the following way.
\begin{Cor}
\label{Cor:CD}
Let $A_1, \ldots, A_k$ be subsets of $\Z_p$, and suppose that $\sum_{i=1}^k
(|\Sigma(A_i)|-1)\geq p-1$. Then $\Sigma(\bigcup A_i)=\Z_p$.
\end{Cor}
\begin{proof}
We have 
\[
|\Sigma(\bigcup A_i)| = |\Sigma(A_1) + \dots + \Sigma(A_k)| \geq \min(1+\sum_{i=1}^k
(|\Sigma(A_i)|-1), p) = p.
\]
\end{proof}

The following result was proven by Olson \cite[Theorem~2]{Olson}.
\begin{Lem}
\label{Lem:Olson}
Let $A\subseteq\Z_p$ be a set with all elements distinct and $|A|=s$. Suppose
that for all $a\in A$, $-a\not\in A$; in particular, $0\not\in A$. Then we have
\[
|\Sigma(A)|\geq\min(\frac{p+3}{2}, \frac{s(s+1)}{2}+\delta),
\]
where
\[
\delta=\begin{cases} 1, & s\equiv 0\pmod{2}\\
0, & s\equiv 1\pmod{2}
\end{cases}.
\]
\end{Lem}

As can be seen by $A=\{1, \ldots, k\}$, this estimate is optimal up to
the value of $\delta$ for odd $k$. This deficiency causes some trouble
in our treatment of small primes, which motivated us to prove the
following using computer calculations \cite{trio3}.
\begin{Lem}\label{Lem:OlsonFmc}
Let $A\subseteq\Z_p$ be a set with all elements distinct and $|A|=s \le 7$.
Suppose that $A$ is zero-sum free. Then
$|\Sigma(A)|\geq \frac{s(s+1)}{2}+1$.
\end{Lem}

The following is a simple consequence of the Lemma of Olson.
\begin{Lem}
\label{Lem:OlsonSize}
Let $A\subseteq\Z_p$ be a set consisting $n + 1$ different elements,
or a set consisting of $n$ different elements and not containing $0$. Then
$|\Sigma(A)|\geq\min(p, \frac{n(n+2)}{4}-1)$.
%Moreover $|\Sigma(A)| \ge \min(p,\ell)$ for
%$(n,\ell) \in \{(3, 5), (4, 7), (5, 9)\}$.
\end{Lem}
\begin{proof}
If $A$ contains $0$, then remove that element. Now partition $A$ into
two sets $B$ and $B'$ with $\lfloor\frac{n}{2}\rfloor$ and
$\lceil\frac{n}{2}\rceil$ which both satisfy the prerequisites of
Lemma~\ref{Lem:Olson}. Using this and Cauchy-Davenport, we get
\[
\Sigma(A) \ge
\begin{cases}
\min\left(p, \frac{n}{2}(\frac{n}{2} + 1) - 1\right)
& \text{if $n$ is even} \\
\min\left(p, \frac{(n-1)(n+1)}{8} + \frac{(n+1)(n+3)}{8} + 1 - 1\right)
& \text{if $n$ is odd.}
\end{cases}
\]
Both cases imply the claim.
%first claim. The table for $n=3,4,5$ is obtained in the
%same way by computing more precisely and using $|\Sigma(B)| \ge 2$ in the case
%$m = 1$.
\end{proof}

The following is due to Dias da Silva and Hamidoune \cite{DdaSH}.
\begin{Lem}
\label{Lem:ksubsets}
Let $A\subseteq\Z_p$ be a set, $k$ an integer in the range $1\leq
k\leq |A|$. Then we have
\[
\left|\Sigma_k(A)\right|\geq\min(p, k(|A|-k)+1).
\]
In particular, if $|A|\geq \ell:=\lfloor\sqrt{4p-7}\rfloor+1$, and
$k=\lfloor \ell/2\rfloor$, then $\Sigma_k(A)=\Z_p$.
\end{Lem}

The next result is a special case of a theorem due to Gao and
Geroldinger \cite{Gaop}.

\begin{Lem}
\label{Lem:Gaop}
Let $A\subseteq\Z_p^2$ be a zero-sum free subset with $2p-2$ elements. If $x,
y\in A$, then they are either the same element of $\Z_p^2$, or they are
linearly independent.
\end{Lem}

The following lemma says that to check that a set $A$ satisfies property B,
it is sufficient to check that all its elements lie in a subgroup and one
coset of that subgroup.

\begin{Lem}
\label{Lem:coset}
Let $A\subseteq\Z_n^2$ with $|A|=2n-2$ be a zero-sum subset such that
there exists a subgroup $H<\Z_n^2$, $H\cong\Z_n$, and an element
$a\in\Z_n^2$, such that $A\subseteq H\cup a+H$. Then $A$ contains an
element with multiplicity $\geq n-2$.
\end{Lem}
\begin{proof}
Suppose that no element occurs $n-2$ times in $A$. Set $s=|A\cap H|$, $t=|A\cap
(x+H)|$. If $s\geq n$, then $H\cap A$ contains a zero-sum, hence, we
have $s\leq n-1$ and therefore $t=n+k$ with $k\geq -1$. Using 
Lemma~\ref{Lem:Onedim}, we find that there are at least $k+1$
distinct elements in $H$ representable as sums of elements from $A\cap
(a+H)$, none of which is zero, and there are at least $s$ non-zero
elements representable by elements in $A \cap H$. Since $(k+1)+s=n-1$,
we find that either there is some element $b\in H$ which is
representable by elements in $A\cap (a+H)$, such that $-u$ is
representable by elements in $A\cap H$, which would yield a zero-sum,
or we have equality in both estimates, that is, all elements in $A\cap
H$ are equal, and either $k \leq 0$ or there are only 2 distinct
elements in $A\cap (a+H)$. If $k \leq 0$, then $s \geq n - 2$ and
$B(d)$ holds. Otherwise, up to linear equivalence, $A$ is of the form
$\{(1, 0)^k, (0, 1)^\ell, (t, 1)^m\}$ with $1\leq
t\leq\frac{n}{2}$. If $t=1$, we have the zero-sum $m\cdot (1,
1)+(n-m)\cdot(1, 0)+(n-m)\cdot(0, 1)$, since 
\[
\min(k, \ell) = k+\ell-\max(k, \ell) \geq k+\ell - (n-3) =
(2n-2-m)-(n-3) > n-m.
\]
Otherwise consider the set $U=\{(-s, 0):1\leq s\leq k\}$ of inverses
of elements representable as non-zero subsums of $A\cap H$, and the
set $V=\{(\nu t, 0):n-\ell\leq \nu\leq m\}$ of elements in $H$
representable by elements in $H\cap x+H$. 
Since $A$ is zero-sum free, we have $0\not\in V$, and $U$ and $V$ are
disjoint. Since $|U|+|V|=n-1$, this implies that $U\cup
V=H\setminus\{0\}$. Suppose that $t>k$. Then $(-t, 0)\in V$, but $(0,
0)\not\in V$, thus, $(m+1)t\equiv 0\;(n)$. Moreover, $(-1, 0)\in V$,
which implies that $t$ and $n$ are coprime, thus, the congruence
$(m+1)t\equiv 0\;(n)$ implies $m\equiv -1\;(n)$. However, this
contradicts the assumption that $1\leq m\leq n-3$. If, on the other hand,
$t\leq k$, we have $(-k-1, 0)\in V$, but $(t-k-1, 0)\not\in V$, which implies
$mt\equiv -k-1\;(n)$, and $(-k-2, 0)\in V$, but $(t-k-2, 0)\not\in V$ implies
$mt\equiv -k-2\;(d)$; thus, either $-k-2\equiv 0\;(n)$, which is impossible for
$1\leq k\leq n-3$, or $t-k-2=k-1$, which leads to the case $t=1$ already
dealt with. Hence, in any case we obtain a zero-sum, and our statement
is proven.
\end{proof}
The next follows from a result of Gao and
Geroldinger \cite[Theorem~3.4]{GGComb}. 
\begin{Lem}
\label{Lem:GGmult}
Let $S\subseteq\Z_p$ be a subset with $|S|\geq\frac{p}{4W}$, where
$W\geq 2$ is an integer and $p\geq 64 W^2$. If every element in $S$
has multiplicity $\leq\frac{p}{40W^2}$, then $S$ contains a zero-sum.
\end{Lem}
\begin{Lem}
\label{Lem:compactDiscrete}
Suppose that $A\subset \Z_p$ satisfies $A + [0,m] = \Z_p$.
Then there is a already subset $A' \subset A$ of cardinality
$|A'| \le 2\lceil \frac{p+1}{m+2}\rceil - 1$ satisfying
$A' + [0,m] = \Z_p$.
\end{Lem}
\begin{proof}
Without loss we can assume $0 \in A$.
Then define a sequence $a_i \in \N$ as follows:
Set $a_1 = 0$, and choose $a_{i+1} \in a_i + \{1,\dots,m+1\}$
maximal such that $a_{i+1} \mod p \in A$ (which is possible by assumption).
For any $i$ we have $a_{i+2} - a_{i} \ge m + 2$, as
otherwise $a_{i+1} - a_{i}$ would not have been maximal,
so $a_{2k - 1} \ge (m+2)(k-1)$ for $k\ge 1$.
We set $k = \lceil \frac{p+1}{m+2}\rceil$
and $A' = \{a_1, \dots, a_{2k - 1}\}$. Then $A' + [0,m] = \Z_p$,
as
\[
a_{2k - 1} + m \ge (m+2)(\lceil \frac{p+1}{m+2}\rceil-1) + m\ge p - 1 .
\]
\end{proof}
The previous Lemma can be applied to give the following, which proves
to be useful if we have many different elements in $A$.
\begin{Lem}
\label{Lem:CompAppl}
Let $A \subset \Z_p^2$ be a subset,
and suppose that $B := \{(1, 0)^{m_1}, (0, 1)^{m_2}\} \subset A$.
Suppose moreover that we can partition $A\setminus B$ into
two sets $U, V$, such that $\Sigma(\pi_2(U)\cup \{1^{m_2}\})=\Z_p$, and
$|\Sigma(\pi_1(V))|> (2\lceil \frac{p+1}{m_2+2}\rceil - 1)\cdot(p-m_1-1)$. Then $A$ contains a zero-sum.
\end{Lem}
\begin{proof}
Applying Lemma~\ref{Lem:compactDiscrete} to
$\Sigma(\pi_2(U))$ (with $m = m_2$) yields a
set $W \subset \Sigma(\pi_2(U))$ with $W + \Sigma(\{(0, 1)^{m_2}\}) = \Z_p$
and $|W| \le 2\lceil \frac{p+1}{m_2+2}\rceil - 1$.
Then for each element
$s\in\Sigma(V)$ there is some index $w \in W$, such that
$\pi_2(s+w)\in[n-m_2, n]$. Hence, we either obtain a zero-sum, or
$\pi_1(s+w)\in[1, p - m_1 - 1]$. If this holds true for all $s\in\Sigma(V)$,
then $\pi_1(\Sigma(V))\subseteq[1, p - m_1 - 1]-\pi_1(W)$.
However, the right hand set contains at most
$(2\lceil \frac{p+1}{m_2+2}\rceil - 1)\cdot (p - m_1 - 1)$ elements,
hence our claim follows.
\end{proof}

\section{The two largest multiplicities of a zero-sum free set in
  $\Z_p^2$} 
\label{sect:2MaxMult}

In this section, we prove \refmm{2max}.

Let $m_1, m_2$ be the two largest multiplicities, and set
$k_i=p-m_i$. We do not assume $m_1\geq m_2$ in this section, in this
way we obtain more symmetry.

We will repeatedly use the
following argument, which for the sake of future citation we formulate
as a lemma. 
\begin{Lem}
\label{Lem:replace}
Let $A$ be a zero-sum free set, $E\subset A$,
and suppose that $\sum_{e\in E}e=k\cdot a$ for some $a \in \Z_p^2$
and some $k \in \N$.
\begin{enumerate}
\item
If $\{a^{k-1}\} \subseteq A\setminus E$, then 
$A\cup\{a^k\}\setminus E$ is zero-sum free.
\item
If $|A| = 2p-2$ and $\{a^{\min(k-1,\lceil p/2 \rceil - 1)}\} \subseteq A\setminus E$
then $|E| \ge k$.
\end{enumerate}
\end{Lem}
\begin{proof}
(1)
Write $A=A_1\cup E\cup \{a^{k-1}\}$. Then we have
\begin{eqnarray*}
\Sigma(A) & = & \Sigma(A_1) + \Sigma(E) + \Sigma(\{a^{k-1}\})\\
 & \supseteq & \Sigma(A_1) + \{0, k\cdot a\} + \Sigma(\{a^{k-1}\})\\
 & = & \Sigma(A_1) + \Sigma(\{a^{2k-1}\})\\
 & = & \Sigma(A_1\cup\{a^{2k-1}\}).
\end{eqnarray*}
Hence, $\Sigma(A)\supseteq\Sigma(A\cup\{a^k\}\setminus E)$, and since
the larger set does not contain 0, the same holds true for the smaller
one.

(2)
If $\{a^{k-1}\} \subseteq A\setminus E$ this follows from the first part.
Otherwise $k-1 > \lceil p/2 \rceil - 1$ and
$\{a^{\lceil p/2 \rceil - 1} \} \subseteq A\setminus E$.
But then $E \cup \{a^{p-k}\}$, which has sum zero, is a subset of $A$:
\[
p-k \le p - \lceil p/2 \rceil - 1 \le \lceil n/2 \rceil - 1.
\]
\end{proof}

We now fix coordinates in such a way that $(1, 0)$ occurs with
multiplicity $m_1$, and $(0, 1)$ with multiplicity $m_2$ in $A$.
Note that in particular, by Lemma~\ref{Lem:Gaop} $A$ does not contain
any element $(k,0)$ or $(0,k)$ for $k \ne 1$.

Denote by $\pi_1$ the projection onto $\langle(1, 0)\rangle$
and by $\pi_2$ the
projection onto $\langle(0, 1)\rangle$.

\begin{Lem}
\label{Lem:2n-9Diff}
Suppose we have $B=\{a_1, \ldots, a_k, b_1, \ldots, b_\ell\}\subset
A\setminus \{(1,0)^{m_1}\}$
for some $k,\ell \ge 1$, with $y=\pi_2(\sum_i a_i)=\pi_2(\sum_i b_i)$.
If $-y\in\Sigma(\pi_2(A\setminus B))$, 
then we have $|\sum_i \pi_1(a_i)-\sum_i\pi_1(b_i)|\le n-m_1-2$.
In particular, this is true if $k+\ell\leq p-m_1-1$.
The same is true with coordinates exchanged.
\end{Lem}

\begin{proof}
Let $c$ be a sum of elements of $A\setminus (B\cup \{(1,0)^{m_1}\})$
with $\pi_2(c) = -y$.
Then $c + \sum_i a_i$ and $c +\sum_i b_i$ both are of the form $(x,0)$.
Such elements can be completed to a zero-sum by copies of $(1,0)$
unless $m_1 < x < n$. The statement follows.

If $k+\ell\leq n-m_1-1$, then $|A\setminus (B\cup \{(1,0)^{m_1}\})| \ge n-1$,
so $\Sigma(\pi_2(A\setminus B))$ contains the whole of $\langle(0,1)\rangle$.
\end{proof}

Our argument will have a recursive structure. For $k_1, k_2\geq 3$ denote
by $B(p, k_1, k_2)$ the statement that there does not exist a zero-sum
free set $A\subseteq\Z_p^2$ with $|A|=2p-2$ and maximal multiplicities
$p-k_1$, $p-k_2$. Note that this statement is false, if one of $k_1,
k_2$ equals 1 or 2, while it is trivially true if one of $k_1, k_2$ is
$\leq 0$. When proving $B(p, k_1, k_2)$ for some pair $(k_1, k_2)$, we
may assume that this statement is already proven for all pairs $(k_1',
k_2')$ with $k_1+k_2>k_1'+k_2'$, such that none of $k_1', k_2'$ equals
1 or 2.

\begin{Lem}
\label{Lem:2n-9locate}
Let $A\subseteq\Z_p^2$ be a zero-sum free set with $|A|=2p-2$, and
suppose that $A$ contains elements with multiplicities $p-k_1$,
$p-k_2$, where $3 \le k_1, k_2\leq p/3$. Then all elements of $A$ different
from $(1, 0)$ and $(0, 1)$ are of the form $(x,y)$ with $1\leq x\leq k_1-2$,
$1\leq y\leq k_2-2$, the form $(p-x, y)$ with $1\leq x\leq k_1-2$,
$1\leq y\leq k_2-1$, or of the form $(x, p-y)$ with $1\leq x\leq k_1-1$,
$1\leq y\leq k_2-2$.
\end{Lem}

\begin{proof}
Suppose that $(x, y)\in A$ with $1\le y<k_2$. Our aim is to show that
$|x|\leq k_1-2$. (Together with the same argument with coordinates exchanged,
this implies the lemma.) We apply
Lemma~\ref{Lem:2n-9Diff} to the sum 
$\pi_2(y\cdot(0, 1))=\pi_2((x, y))$, and deduce that
\[
|x| = \left|\pi_1((x, y)) - \sum_{i=1}^x \pi_1((0, 1))\right| \leq k_1-2,
\]
provided that $-y\in\Sigma(\pi_2(A\setminus\{(0, 1)^y, (x, y)\}))$.
Hence, from now on we assume that this is not the case.

If there were an element $a\in A$ with $k_2\leq
\ii(\pi_2(a))\leq p-y$, then this element together with $(0,
1)^{p-k_2-y}$ would represent $-y$, hence, there is no element in this
range. Denote by $B$ the set of all $a\in A\setminus\{(1, 0)^{p-k_1}, (0,
1)^{p-k_2}, (x, y)\}$ 
with $\ii(\pi_2(a))>p-y$, and by $C$ the set of all $a\in A\setminus\{(1, 0)^{p-k_1}, (0,
1)^{p-k_2}, (x, y)\}$ with $i(\pi_2(a))<k_2$. Then $-y$ is representable as
subsum of $\pi_2(B)$ together with a certain multiple of $(0, 1)$, if
$\sum_{b\in B} p-i(\pi_2(b)) \geq y$, and $-y$ is representable as
subsum of $\pi_2(C)$ together with a certain multiple of $(0, 1)$, if
$\sum_{c\in C} i(\pi_2(c)) \geq k_2$; in particular $|B| \le y - 1$
and $|C| \le k_{2}-1$.

We now form the sum $s$ of all elements in $B$. Then we have
$n-\ii(\pi_2(s))=\sum_{b\in B}(n-\ii(\pi_2(b))\leq y-1\leq n/3$, hence,
if $\sum_{b\in B}\ii(\pi_1(b))\geq k_1$, we can add a certain multiple
of $(1, 0)$ and $(0, 1)$ to $s$ and obtain a zero-sum. In particular,
$|B|\leq k_1-1$.

This
implies $|C|\geq k_2-2$, as $|B| + |C| = k_{1} + k_{2} - 3$. Since
$\sum_{c\in C}\ii(\pi_2(c))\leq k_2-1$, we deduce that $C$ contains at
most one element $c_0$ with $\pi_2(c_0)\neq 1$, and, if it exists,
this element satisfies $\pi_2(c_0)=2$.

Similarly, $|C|\leq k_2-1$ implies
$|B|\geq k_1-2 \ge 1$, and therefore $B$ contains at most one element
$b_0$ with $\pi_1(b_0)\neq 1$, and this element satisfies
$\pi_1(b_0)=2$.

In particular, $B$ and $C$ are both non-empty.

Suppose there exist elements $b\in B$, $c\in C$ with $b\neq b_0$,
$c\neq c_0$. Then $b+c$ can be combined with certain multiples of $(1,
0)$ and $(0, 1)$ to a zero-sum, unless $\pi_1(c)\in[1, k_1-2]$. 

Consider again the sum $s$ of all elements in $B$. This sum satisfies
$\pi_1(s)\in\{k_1-2, k_1-1\}$, $\pi_2(s)\in[n-t+1, n-|B|]$. Hence,
adding $c$ we obtain a zero-sum, unless $\pi_1(c)=1$ and
$\pi_1(s)=k_1-2$. In particular, $|B|=k_1-2$, $|C|=k_1-1$, and $b_0,
c_0$ do not exist, that is, all elements in $C$ are equal to $(1,
1)$. If $x\in[p-|C|, p]$, we add $p-x$ copies of $(1, 1)$ to $(x, y)$
to obtain $(0, p+y-x)$ as the sum of $p-x+1$ elements. Hence, we can
replace $p-x+1$ elements of $A$ by $p-x+y$ copies of $(0, 1)$, which
gives a zero-sum, unless $y=1$, which is impossible, since $|B|\leq
y-1$. If $x\not\in[p-|C|, p]$, we add all copies of $(1, 1)$ to $(x,
y)$ and obtain an element $s$ with $\pi_2(s)\in[k_1-1+|C|,
p]\subseteq[k_1, p]$, $\pi_1(s)\in[y+k_2-1, p]\subseteq[k_2, p]$,
hence, $s$ can be combined with a certain number of copies of $(1, 0)$
and $(0, 1)$ to a zero-sum. Thus, the assumption that both $B$ and $C$
contain elements different from $b_0, c_0$ was wrong.

Suppose that $C=\{c_0\}$. Then $k_2=3$ and $|B|=k_1-1$, therefore
$k_1=|B|+1\leq y\leq k_2-1=2$, contrary to the assumption $k_1\geq
3$. If $B=\{b_0\}$, then $k_1=3$, $|C|=k_2-1$, and all elements in
$C$ satisfy $\pi_2(c)=1$. If $C$ contains an element different from
$(1, -1)$, we add this element to $b_0$ and obtain a zero-sum, hence,
$C=\{(1, -1)^{k_2-1}\}$. If $\pi_2(b_0)\neq -1$, consider $b_0+2(1,
-1)$. This element can be combined with a certain multiple of $(0, 1)$
to a zero-sum. If $\pi_2(b_0)=-1$, we can replace $b_0$ and one copy
of $(0, 1)$ by 2 copies of $(1, 0)$, hence, we obtain a zero-sum free
set $A'$ of cardinality $2p-2$ containing an element of multiplicity $p-1$,
that is, all elements of $A$ different from $b_0$ and $(1, 0)$ are of
the form $(u, 1)$, in particular, $y=1$. But then $-y=\pi_2(b_0)$ is
representable, and the proof is complete.
\end{proof}

% The following image does not display correctly in xdvi (but in the ps file
% it works)
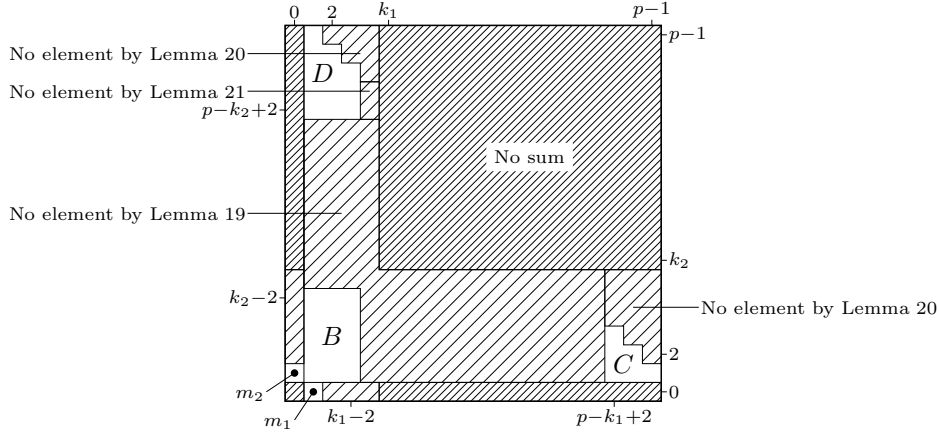
\begin{figure}
\begingroup
\psset{unit=2.5mm,dimen=middle,linewidth=0.1mm,hatchwidth=0.1mm}
\SpecialCoor
% In the following definitions, the space at the end is important:
\def\gr{20 }   % total size of image
\def\kx{5 }    % k_1
\def\ky{7 }    % k_2

\def\kxm{\kx 1 sub }
\def\kym{\ky 1 sub }
\def\kxr{\gr \kx sub 2 add }
\def\kyr{\gr \ky sub 2 add }

\def\xmark#1#2{\rput(! #1 0.5 add 0){\psline(0,-.3)(0,0)\rput[t](0,-.4){$\scriptstyle{#2}$}}}
\def\ymark#1#2{\rput(! 0 #1 0.5 add){\psline(-.3,0)(0,0)\rput[r](-.4,0){$\scriptstyle{#2}$}}}
\def\xmarkt#1#2{\rput(! #1 0.5 add \gr){\psline(0,.3)(0,0)\rput[b](0,.4){$\scriptstyle{#2}$}}}
\def\ymarkr#1#2{\rput(! \gr #1 0.5 add){\psline(.3,0)(0,0)\rput[l](.4,0){$\scriptstyle{#2}$}}}
% comm(#1)(#2){blah}: put comment "blah" at the left of #2, pointing to #1; (#2) may be omitted
% same for commr, but comment is put at the right.
\makeatletter
\def\comm(#1){\@ifnextchar({\comm@(#1)}{\comm@(#1)(#1 exch pop -2 exch)}}
\def\commr(#1){\@ifnextchar({\comm@r(#1)}{\comm@r(#1)(#1 exch pop \gr 2 add exch)}}
\def\comm@(#1)(#2)#3{\psline(! #1)(! #2)\rput[r](! #2 exch 0.1 sub exch){#3\strut}}
\def\comm@r(#1)(#2)#3{\psline(! #1)(! #2)\rput[l](! #2 exch 0.1 add exch){#3\strut}}
\makeatother
\begin{pspicture}(-3,-1.2)(22.2,21)
\psframe[linewidth=0.2mm](0,0)(\gr,\gr)
\psline(0,1)(\gr,1)
\psline(1,0)(1,\gr)
% No element:
{\psset{fillstyle=hlines,hatchsep=1.2mm}
  %y=0,x\ne1
  \psframe(2,0)(! \kx 1)
  %x=0,y\ne1
  \psframe(0,2)(! 1 \ky)
  % Lem:2n-9locate:
  \pspolygon(! 1 \kym)(! \kxm \kym)(! \kxm 1)(! \kxr 1)(! \kxr \ky)(! \kx \ky)(! \kx \kyr)(! 1 \kyr)
  % Lem:stairs C:
  \pspolygon(! \gr 2)(! \gr 1 sub 2)(! \gr 1 sub 3)(! \gr 2 sub 3)%
            (! \gr 2 sub 4)(! \gr 3 sub 4)(! \kxr \ky)(! \gr \ky)
  % Lem:stairs D:
  \pspolygon(! 2 \gr)(! 2 \gr 1 sub)(! 3 \gr 1 sub)(! 3 \gr 2 sub)%
            (! 4 \gr 2 sub)(! 4 \gr 3 sub)(! 5 \gr 3 sub)(! \kx \gr)
  % Lem:symmetry:
  \psframe(! \kx 1 sub \kyr)(! \kx \kxr)
}
% No sum:
{\psset{fillstyle=hlines,hatchsep=.6mm}
  \psframe(! \kx \ky)(\gr,\gr)
  \psframe(! 0 \ky)(1,\gr)
  \psframe(! \kx 0)(\gr,1)
  \psframe(0,0)(1,1)
}
\pscircle*(1.5,0.5){0.2}\comm(1.5 0.5)(.5 -1){$\scriptstyle{m_1}$}
\pscircle*(0.5,1.5){0.2}\comm(0.5 1.5)(-1 .5){$\scriptstyle{m_2}$}

\comm(3 10){{\scriptsize No element by Lemma~\ref{Lem:2n-9locate}}}
\comm(\kx 1 sub \gr 1.5 sub){{\scriptsize No element by Lemma~\ref{Lem:stairs}}}
\commr(\gr 1.5 sub \kym 1 sub){{\scriptsize No element by Lemma~\ref{Lem:stairs}}}
\comm(\kx 0.5 sub \kxr 0.5 sub){{\scriptsize No element by Lemma~\ref{Lem:symmetry}}}
\rput*(13,13){\scriptsize{No sum}}
\xmark{\kxm 1 sub }{k_1-2}
\ymark{\kym 1 sub }{k_2-2}
\xmark{\kxr }{p-k_1+2}
\ymark{\kyr }{p-k_2+2}
\xmarkt{0}{0}
\ymarkr{0}{0}
\xmarkt{2}{2}
\ymarkr{2}{2}
\xmarkt{\kx }{k_1}
\ymarkr{\ky }{k_2}
\xmarkt{\gr 1 sub }{p-1}
\ymarkr{\gr 1 sub }{p-1}
\rput(2.5,3.5){$B$}
\rput(18,2){$C$}
\rput(2,17.5){$D$}
\end{pspicture}
\endgroup
\caption{What we know about $A \setminus \{(1,0)^{m_1},(0,1)^{m_2}\}$}
\label{fig:BCD}
\end{figure}

Note that these three rectangles are disjoint. From now on we will
denote the set of points in $A\setminus\{(1,0)^{m_1},(0,1)^{m_2}\}$
of the form $(x, y)$ by $B$, the set
of points of the form $(p-x, y)$ by $C$, and the set of points of the
form $(x, p-y)$ by $D$ ($x< k_1, y< k_2$).
%Moreover, in cases with $k_1=k_2$ we will
%always assume that $|C|\geq|D|$.

Our next result further restricts elements in $C$ and $D$. At this
place we use the induction on $k_1, k_2$ for the first time.

\begin{Lem}
\label{Lem:stairs}
Let $A\subseteq\Z_p^2$ be a zero-sum free set with $|A|=2p-2$, and
suppose that $(1, 0)$, $(0, 1)$ are the elements with highest
multiplicity $p-k_1$, $p-k_2$, respectively. Let
$A\setminus\{(1,0)^{m_1},(0,1)^{m_2}\} = B \cup C \cup D$ be the decomposition
as above. Then $C$ does not contain an element $(p-x, y)$ with $x<y$,
and $D$ does not contain an element $(x, p-y)$ with $y<x$.
\end{Lem}
\begin{proof}
Suppose that $(p-x, y)\in C$ with $y>x$. 
Apply Lemma~\ref{Lem:replace} to $E:= \{(p-x, y), (1,0)^x\}$.
We conclude that the set
$A^* = A \setminus E \cup \{(0,1)^y\}$ is zero-sum free. If $y>x+1$,
then $|A^*|>2p-2$, which is impossible. If $y=x+1$, then 
$A^*$ has cardinality $2n-2$ and maximal multiplicities $p-k_1-y+1$,
$p-k_2+y$, hence, by our inductive hypothesis we obtain
$p-k_2+y\in\{p-2, p-1\}$. Thus all elements $a$ in $A$ different from
$(p-x, y)$ satisfy $\pi_1(a)\in\{0, 1\}$. If $B$ is non-empty, say,
$b=(1, z)\in B$, we can apply Lemma~\ref{Lem:replace} to $E:=
\{(p-y+1, y), (1, z), (1,0)^{y-2}\}$, and obtain a
contradiction. Hence, $|D|\geq k_1+k_2-3\geq k_1$. The sum $s$ of
$k_1$ elements of $D$ satisfies $\pi_1(s)=k_1$, hence, we either
obtain a zero-sum, or $\pi_2(s)\in[1, k_2-1]$. The latter is only
possible if the average value of $p-\pi_2(d)$ taken over all elements
$d\in D$ is larger than 2. Hence, we can choose a subset $D'$ of $D$
with sum $s$ satisfying $\pi_1(s)\in[1, y/2]$, $\pi_2(s)\in[k_2,
p-y]$. But then $s+(p-y+1, y)$ can be combined with some multiples of
$(1, 0)$ and $(0, 1)$ to a zero-sum.
\end{proof}

Now, we can remove the apparent asymmetry in
Lemma~\ref{Lem:2n-9locate}.
\begin{Lem}\label{Lem:symmetry}
$C$ does not contain an element $c$ with $\pi_2(c)=k_2-1$, and $D$
does not contain an element with $\pi_1(d)=k_1-1$.
\end{Lem}
\begin{proof}
By symmetry, it suffices to prove the statement for $D$.

Suppose that $d=(k_1-1, p-y)\in D$. By
Lemma~\ref{Lem:2n-9locate} and \ref{Lem:stairs} we have $k_1-1\leq
y\leq k_2-2$. Suppose that $D$ 
contains another element $(x', n-y')$. Then we obtain the zero-sum $(x',
p-y')+(k_1-1, p-y)+(n-x'-k_1+1)\cdot(1, 0)+(y'+y)\cdot(0, 1)$. Next,
suppose that $B$ contains an element $(x', y')$. If $y'\leq t$, we
obtain the zero-sum $(x', y')+(k_1-1, p-y)+(p-x'-k_1+1)(1,
0)+(y-y')(0, 1)$, thus, all elements $b \in B$ satisfy $\pi_2(b)\geq
y+1\geq k_1$.

Let $s$ be the sum of all elements in $B$ and $C$. If
$\pi_2(s)\geq k_2+y$, we can choose some subset sum $s'$ satisfying
$\pi_2(s')\in[k_2+y, 2k_2+y)$. Then $\pi_2(s'), \pi_2(s'+d)\in[k_2,
2k_2]$, hence, we either get a zero-sum by adding a certain multiple of
$(1, 0)$ and $(0, 1)$, or $\pi_1(s'), \pi_1(s'+d)\in[1, k_1-1]$. But
this is impossible since $\pi_1(s'+d)=\pi_1(s')+k_1-1$. Hence, we
obtain $\pi_2(s)<k_2+y$.

Denote by $C_1$ the set of all $c\in C$ with
$\pi_2(c)=1$, and $C_2$ the set of all $c$ with $\pi_2(c)\geq 2$. Then
\[
\pi_2(s)\geq (y+1)|B|+|C|+|C_2|\geq y|B|+|C_2|+k_1+k_2-3,
\]
thus, for $|B|\geq 1$ we obtain the inequality $k_1-3<0$, which is
false. Hence, $B=\emptyset$, and $|C_2|\leq y+3-k_1$, thus,
\[
|C_1| = k_1+k_2-3-|C_2|\geq 2k_1+k_2-y-6\geq k_1-1.
\]
Choose a subset $C'\subseteq C_1$ with $\sum_{c\in C'} p-\pi_1(c)\geq
k_1-1$ and $|C'|$ minimal with this property, and let $s$ be the sum
of all elements of $C'$. Then $\pi_1(s+c)\in[p-k_1, p]$, and
$\pi_2(s+c)\in[p-y+1, p]$, hence, $s+c$ can be combined with certain
multiples of $(1, 0)$ and $(0,1)$ to a zero-sum.
\end{proof}

\begin{Lem}
\label{Lem:Bexists}
Suppose that $B$ is empty. Then there is a zero-sum.
\end{Lem}
\begin{proof}
If $C$ contains an element with $\pi_2(c)=1$, and $D$ contains an
element with $\pi_1(d)=1$, their sum can be combined with a certain
number of copies of $(1, 0)$ and $(0, 1)$ to give a zero-sum. Hence,
we may assume that all elements in $D$ satisfy $\pi_1(d)\geq
2$. Suppose that $\sum_{c\in C}\ii(p-\pi_1(c))\geq k_1 - 1$, and that
$\sum_{d\in D}\ii(p-\pi_2(d))\geq k_2 - 1$. Then we can choose a subset $C' \subset C$
such that the sum $s_C$ of all elements in $C'$ satisfies
$\pi_1(s_C)\in[k_1, p-k_1+1]$. We may suppose $\pi_2(s_C)\leq k_2-1$; otherwise
we get a zero-sum. Analogously, we may choose a subset $D' \subset D$ whose
sum $s_D$ satisfies $\pi_1(s_D)\leq k_1-1$, $\pi_2(s_D)\in[k_2, p-k_2+1]$. Hence, $s_C+s_D$
yields a zero-sum.

Suppose that $\sum_{d\in D}\ii(p-\pi_2(d))<k_2$. Suppose that $D$ is
non-empty, and that $(x, p-y)\in D$ with $x$ minimal. Then
$|D|\leq\frac{k_1-1}{d}$, and in particular $|C|\geq
k_2-1+(1-1/d)(k_1-1)$. If $C$ contains $d$ elements with $\pi_2(c)=1$,
adding some of these elements to $(x, y)$ yields an elements which can
be combined with some copies of $(1, 0)$ and $(0, 1)$ to a zero-sum.
Consider a subset $C'\subseteq C$, such that 
$s=\sum_{c\in C'}c$ satisfies $\pi_2(s)\geq k_2$, and that no proper
subset of $C'$ 
satisfies this property. Then $\pi_2(s)\in[k_2, 2k_2]$, thus we either
obtain a zero-sum or $\pi_1(s)\in[1, k_1-1]$, hence the average value
$\mu$ of $\frac{p-\pi_1(c)}{\pi_2(c)}$ taken over all elements in $C$
satisfies $\mu\geq\frac{p-k_1+1}{k_2}$. If $\sum_{c\in C\setminus
  C'}p-\pi_1(c)\geq k_1-1$, we can choose a subset $C''$ disjoint to
$C'$ with sum $s'$ satisfying $\pi_1(s')\in[p-2k_1, p-k_1]$,
$\pi_2(s')\in[1, k_1+k_2-2]$. Hence, $s+s'$ can be combined with
certain copies of $(1, 0)$ and $(0, 1)$ to a zero-sum. In particular,
$\sum_{c\in C}p-\pi_2(c)<p$. Hence, we obtain
\[
p>\sum_{c\in C}p-\pi_2(c)\geq (2|C|-d)\mu\geq
(2k_2-2+(2-2/d)(k_1-1)-d)\frac{p-k_1+1}{k_2}.
\]
Using $p-k_1>\frac{2}{3}p$, we see that this yields a contradiction,
unless
\[
\frac{k_2}{2}+(2-2/d)(k_1-1)-d-2<0.
\]
For $d=2$ this yields $2k_1+k_2<12$, which is covered by \refmc{2max}, while for $d\geq 3$ we use the bound $d\leq
k_1-1$ and obtain $2k_1+3k_2<14$, which is impossible in view of $k_1,
k_2\geq 3$.

If $D=\emptyset$, the same argument yields $p>|C|\mu$, thus
$p>(k_1+k_2-2)\frac{p-k_1+1}{k_2}$, thus
$(k_1-2)\frac{p-k_1+1}{k_2}\leq k_1-2$. However, this contradicts the
assumption $k_1, k_2\leq p/3$.

We may therefore assume that $\sum_{c\in C}(n-\pi_1(c))<k_1$, and
therefore $|D|\geq \max(k_1-1, k_2-1)$. For every subset $D'\subseteq
D$ consisting of $\lceil\frac{k_1}{2}\rceil$ elements we have
$\sum_{d\in D'}(n-\pi_2(d))\geq p-k_2+1$, while $\sum_{d\in D}(n-\pi_2(d))
\leq p-1$, hence 
\[
\frac{|D|}{\lceil\frac{k_1}{2}\rceil}(p-k_2+1)\leq p-1.
\]
From $k_2<p/3$ we obtain $|D|\leq \frac{3(k_1+1)}{4}$, which
contradicts $|D|\geq \max(k_1-1, k_2-1)$, unless $k_1, k_2\leq
7$. However, for $k_1+k_2\leq 14$ our claim follows from \refmc{2max}.
\end{proof}

We can now finish the proof of \refmm{2max}.

If $C$ and $D$ are both empty, then $|B|=k_1+k_2-2$. Set $B_1=\{b\in
B:\pi_1(b)>\pi_2(b)\}$, $B_2=\{b\in B:\pi_1(b)<\pi_2(b)\}$,
$B_3=\{b\in B:\pi_1(b)=\pi_2(b)\}$. Suppose that $\sum_{b\in B_1\cup
  B_3}\pi_1(b)\geq k_1$, and $\sum_{b\in B_2\cup B_3}\pi_1(b)\geq
k_2$. Then we obtain a zero-sum by first choosing a subset of
$B_1\cup B_3$ with sum $s$ minimal subject to the condition
$\pi_1(s)\geq k_1$, and then we add elements from $B_2$ and elements
not yet used from $B_3$ to reach a sum $s'$ with $\pi_1(s')\in[k_1,
2k_1+k_2]$, $\pi_2(s)\in[k_2, 2k_2]$. Without loss we may
assume that $\sum_{b\in B_2\cup B_3}\pi_2(b)< k_2$. Choose a subset $B'$
of $B$ containing $B_2\cup B_3$ with sum $s$ such that $\pi_2(s)\geq
k_2$, and that $\pi_2(s)$ is minimal with respect to these conditions.
Then $\pi_1(s)\in[1, k_1-1]$, for otherwise we obtain a
zero-sum. There are at least $k_1-2$ elements in $B_1$ not involved in
this sum, and each element in $B_1$ satisfies $\pi_1(b)\geq 2$, hence,
we can choose a subset $B''$ in the remainder with $\sum_{b\in B'}
\pi_1(b)\geq k_1-1$, and $B''$ minimal with this property. In
particular, $\sum_{b\in B'}\pi_1(b)\in[k_1-1, 2k_1-1]$, and
$\sum_{b\in B'}\pi_2(b)\leq k_2$. Hence, adding the elements in $B'$
and the elements in $B''$, we obtain an element which can be combined
with some copies of $(1, 0)$ and $(0, 1)$ to a zero-sum.

Hence, without loss we may assume that $C$ is non-empty. Fix
elements $b\in B, c\in C$. Consider the sets $\mathcal{S}=\Sigma(\{(1, 0)^{m_1},
(0, 1)^{m_2}, b, c\})$, $\mathcal{S'}=\Sigma(\{(1, 0)^{m_1},
(0, 1)^{m_2}, b-(1, 0), c+(1, 0)\})$. Then
\[
\mathcal{S'} \subset \mathcal{S}\cup\underbrace{\{b+m_1(1, 0)+t(0, 1):0\leq t\leq
m_2\}\cup\{c+t(0, 1):0\leq t\leq m_2\}}_{(*)}
\]\[
\text{and}\quad\underbrace{\{0, b+c\}+\Sigma(\{(1, 0)^{m_1}, (0, 1)^{m_2}\})}_{(**)} \subset S
.
\]
Since $m_1, m_2\geq 2p/3$, we get that $(*)$ is contained in $(**)$, and so $\mathcal{S}'\subseteq\mathcal{S}$.
Hence, if $A$ is zero-sum free,
the set $A'$ obtained by replacing $b$ by $b-(1, 0)$ and $c$ by $c+(1,
0)$ is also zero-sum free. We can repeat this procedure, until one of
$b, c$ is contained in $\langle(0, 1)\rangle$. If the element obtained
in this way is not equal to $(0, 1)$, we can replace it by at least two
copies of $(0, 1)$, which is impossible. If it is equal to $(0, 1)$,
which can only happen if $\pi_1(b)=1$ or $\pi_1(c)=1$,
we can replace one or two elements from $B$ and $C$ by as many copies
of $(0, 1)$, that is, our claim follows from the inductive hypothesis,
unless the resulting set contains an element with multiplicity $\geq
p-2$. Since the element with multiplicity $\geq p-2$ is necessarily
$(0, 1)$, and $(1, 0)\in A$, we find that all elements different from the
elements $b$ and $c$ chosen at the beginning are
contained in $(1, 0)+\langle(0, 1)\rangle$. In particular,
$C=\{c\}$. If $D\neq\emptyset$, we 
obtain in the same way $m_1\geq p-4$, that is, $k_1+k_2\leq 8$, a case
which is covered by our computations.

Hence, it remains to consider the case $|B|=k_1+k_2-3$, $C=\{c\}$,
$k_2\leq 4$. Moreover, since $|B|\geq 2$, we could use any element
$b\in B$ in the argument above and find that all elements in $B$
satisfy $\pi_1(b)=1$. Hence, replacing $c$ by $c+(1, 0)$ and $b$ by
$b-(1, 0)$ yields a zero-sum free set of cardinality $2p-3+\pi_2(b)$,
thus, $B=\{(1, 1)^{k_1+k_2-3}\}$. But then we can form the zero-sum
$(p-\pi_1(c))(1, 1)+c+(\pi_1(c)-\pi_2(c))(0, 1)$, which is possible
since $\pi_2(c)\geq 2$ by Lemma~\ref{Lem:coset}, and
$p-\pi_1(c)\geq\pi_2(c)$ by Lemma~\ref{Lem:stairs}. Hence, the Theorem
is proven. 

\section{The largest multiplicity of a zero-sum free set in
  $\Z_p^2$}
\label{sect:1MaxMult}

In this section let $A\subset\Z_p^2$ be a zero-sum free set with $|A|=2p-2$ and
maximal multiplicity $p-3$. Denote by $m$ the second largest
multiplicity in $A$. Without loss we may assume that $(1, 0)$ occurs
in $A$ with multiplicity $p-3$, and $(0, 1)$ with multiplicity $m$.
Moreover, by \refmc{all} we may suppose $p \ge 29$.
By \refmm{2max}, we get
$p-m > 29/3$, thus $m\leq p-10$.
\begin{Lem}
\label{Lem:n-3Diff}
Suppose that $(x, y), (x', y)\in A$. Then $|x-x'|\leq 1$. Moreover,
there is at most one pair $a, a'$ of elements in $A$ with $a\neq a'$,
$\pi_2(a)=\pi_2(a')$; in particular, the maximal multiplicity of
$\pi_2(A \setminus \{(1,0)^{p-3}\})$ is at most $m+1$.
\end{Lem}
\begin{proof}
The first claim follows from Lemma~\ref{Lem:2n-9Diff}, if we can show
that $-y\in\Sigma(\pi_2(A)\setminus\{y^2\})$, which in turn is implied
by the Cauchy-Davenport-theorem. For the second claim suppose that $a_1,
a_1', a_2, a_2'$ are elements fo $A$ with $a_i\neq a_i'$,
$\pi_2(a_i)=\pi_2(a_i')$. Then we apply Lemma~\ref{Lem:2n-9Diff} to
$a_1+a_2$, $a_1'+a_2'$, where we may assume
$|\pi_1(a_1+a_2)-\pi_1(a_1'+a_2')|=2$. Note that
$S=\pi_2(A\setminus\{a_1, a_2, a_1', a_2', (1, 0)^{p-3}\})$ contains
$p-3$ non-zero elements, hence, 
the Cauchy-Davenport-theorem together with Lemma~\ref{Lem:Olson} imply
that $\Sigma(S)=\Z_p$ unless all elements in $S$ are equal with at most one
exception, that is $\pi_2(A)$ contains some non-zero element $y$ with
multiplicity $\geq p-5>2p/3+2 \ge m + 2$. Using     
the first part of the lemma, we get $\{(x,y)^\ell,
(x+1,y)^{\ell'}\} \subset A$  for some $x \in \Z_p$ and $\ell,\ell' \ge 2$.
Now we replace $a_1, a_2, a_1', a_2'$ by two pairs $(x,y)$, $(x+1,y)$
and do the same argument again. As a result, we get
$A=\{(1, 0)^{p-3}, (x, y)^\ell, (x+1, y)^{\ell'}, a\}$ with
$\ell + \ell' = p$ and both $\le m \le 2p/3$. But this contains
the zero-sum
$\ell\cdot(x, y) + \ell\cdot(1,0) + \ell'\cdot(x+1, y)$.
\end{proof}

\begin{Lem}
\label{Lem:n-3:SmallMult}
If $m\leq\frac{p}{6}$, then $A$ contains a zero-sum.
\end{Lem}
\begin{proof}
It suffices to show that $\pi_2(A\setminus\{(1, 0)^{p-3}\})$ contains three
disjoint zero-sums: these zero-sums generate three elements in
$\langle(1, 0)\rangle$, hence, together with some copies of $(1, 0)$
we obtain a zero-sum in $A$. By Lemma~\ref{Lem:n-3Diff}, we may choose
$a \in A$ such that
$S = \pi_2(A\setminus\{(1, 0)^{p-3}, a\})$ has
maximal multiplicity (at most) $m$.
Then we can split $S$ into subsets of given cardinalities,
each of which having no multiple elements, provided that each
given cardinality
are at most 6. We choose to do this in the following way:
Set $d = \lfloor\frac{p}{3}\rfloor$ and $r = d \bmod 6$.
We form $3\cdot\lfloor\frac{d}{6}\rfloor$ sets of cardinality $6$
and $3$ (possibly empty) sets of cardinality $r$.
Then we group these small sets into three sets
$S_1, S_2, S_3$, each being the union of $\lfloor\frac{d}{6}\rfloor$ subsets
of cardinality 6 and one of cardinality $r$.
If we can show that each $S_i$ contains
a zero-sum, we are done. If one of the small sets contains a zero-sum,
then so does each larger set, hence, we may assume that each of the
small sets is zero-sum free, and we can apply Lemma~\ref{Lem:OlsonFmc}.
Thus $S_{i}$ contains a zero-sum provided that
\[
1+\lfloor\frac{d}{6}\rfloor\cdot 21 + \frac{r(r+1)}{2} \ge p
.
\]
The left hand side is equal to
\[
1+\frac{d-r}{6}\cdot 21 + \frac{r(r+1)}{2}
=1+\frac{7}{2}\cdot\lfloor\frac{p}{3}\rfloor + \frac{r(r-6)}{2}
\ge \frac{7}{6}p - \frac{4}{3} + \frac{r(r-6)}{2}
.
\]
This is minimal for $r = 3$, so the inequality holds
provided that $\frac{1}{6}p \ge \frac{4}{3} + \frac{9}{2}$,
i.e.\ $p \ge 35$.

For $p = 29$ or $31$, we apply the same argument but decompose $S$
differently.
If $p = 29$, then $m \le 4$, and we can choose three $9$-element sets
$S_i$ each one consisting of one set of 7 distinct points and one pair
of distinct points, which suffices. If $p = 31$, then $m \le 5$,
and we obtain three $10$-element sets
consisting of $6$ distinct points plus $4$ distinct points,
which also suffices.
\end{proof}

Define $k=\lceil\frac{p}{m}\rceil$. The introduction of this parameter
turns out to be useful for two reasons: first, it distinguishes
several cases for which we shall use different arguments, and second,
we will apply Lemma~\ref{Lem:CompAppl}, which
involves $k$.
 Note that by Lemma~\ref{Lem:n-3:SmallMult},
only the values $2 \le k \le 6$ are left.

In the present case, the condition on $V$ of Lemma~\ref{Lem:CompAppl}
becomes $|\Sigma(\pi_1(V))| \ge 4k - 1$. Verifying the condition for $U$
is facilitated by the following simple observation. 
\begin{Lem}
\label{Lem:CompHilf}
Let $U\subseteq\Z_p$ be a set satisfying $|u|\leq m$ for all $u\in
U$. Then $\Sigma(\{1^m\}\cup U)=\Z_p$ is equivalent to $\sum_{u\in
  U}|u| \ge p-m-1$. 
\end{Lem}
\begin{proof}
If $x, y, u \in \Z_p$ satisfy $|u|\leq \ii(y-x)$ , then
\[
\{x, x+1, \ldots, y\}+\{0, u\} = \begin{cases}
\{x, x+1, \ldots, y+u\}, & \ii(u)=|u|\\
\{x-u, x-u+1, \ldots, y\}, & \ii(u)=p-|u|.
\end{cases}
\]
Our claim now follows by induction on $|U|$.
\end{proof}
\begin{Lem}
\label{Lem:k2}
Suppose that $k=2$ (i.e. $m \ge \frac{p}{2}$). Then $A$ contains a zero-sum.
\end{Lem}
\begin{proof}
Every subset $V\subseteq A\setminus\{(1, 0)^{p-3}, (0, 1)^m\}$ with
$|V|=6$ satisfies the condition of Lemma~\ref{Lem:CompAppl}. $A$ contains
at most one element $a$ with $\pi_2(a)=1$ different from $(0, 1)$, and
at most two elements with $\pi_2(a)=-1$, hence, putting these elements
into $V$ we may assume that all elements of $U$ satisfy $|u|\geq
2$. Since $m\geq p/2$, we can apply Lemma~\ref{Lem:CompHilf}, and our
claim follows, if 
\[
p-m-1 \le \sum_{u\in U}|\pi_2(u)|\geq 2|U| = 2(p-m-5),
\]
which is true since $p-m\geq 9$.
\end{proof}

\begin{Lem}\label{Lem:k36}
Suppose $3 \le k \le 6$. Then $A$ contains a zero-sum.
\end{Lem}
\begin{proof}
Define $E = A\setminus\{(1, 0)^{p-3}, a\}$, where $a$ is chosen
such that the maximal multiplicity of $\pi_{2}(E)$ is at most $m$.
As $\frac{p}{m} > k-1$,
we can partition $\pi_{2}(E)$ into $\lfloor\frac{p}{k-1}\rfloor$ subsets
$S_{i}$, each one
consisting of $k-1$ distinct elements, and leaving $p \bmod (k-1)$ elements
unused. Let $\ell$ be the number of sets $S_{i}$ containing a zero-sum.

Note first that if $\ell \ge 3$, we are done: each zero-sum comes from a sum
$s$ of elements of $A$ with $\pi_{2}(s) = 0$; together with the elements $(1,0)^{p-3}$,
this yields a zero-sum.
If $\ell < 3$, we apply Lemma~\ref{Lem:CompAppl} to the set
$A'$ which has been obtained from $A$ by removing the pre-image of each set
$S_{i}$ containing a zero-sum, and adding $\ell$ copies of $(1,0)$.
If $A'$ contains a zero-sum, then so does $A$, so we are done if we can
find sets $U, V$ satisfying the prerequisites of the lemma.

Let $m'$ be the multiplicity of $(0,1)$ in $A'$, and set
$k' = \lceil\frac{p+1}{m'+2}\rceil$.
The condition on $V$ is $|\Sigma V| > (2k'-1)(2 - \ell)$;
this is satisfied for any set $V$ with $|V| \ge (2k'-1)(2 - \ell)$.
Note that $k' \le \lceil\frac{p+1}{m}\rceil = \lceil\frac{p}{m}\rceil = k$
as $m' \ge m - 2$ and $m$ does not divide $p$.

Set $\sigma = \frac{(k-1)k}{2}$; by Lemma~\ref{Lem:OlsonFmc},
each set $S_{i}$ has a sumset of cardinality at least $\sigma+1$,
so by Cauchy-Davenport, to get $\Sigma(\pi_{2}(U) \cup \{1^{m'}\}) = \Z_{p}$ it suffices to
ensure that $\pi_{2}(U \cup \{(0,1)^{m'}\})$ contains at least
$\lceil\frac{p-1}{\sigma}\rceil$ of the sets $S_{i}$. Thus we can satisfy
all prerequisites of the lemma if there are at least
$\lceil\frac{p-1}{\sigma}\rceil + \ell$ sets $S_{i}$ in $\pi_{2}(E)$ and
at least $(2k'-1)(2 - \ell)$ additional elements in $A\setminus\{(1, 0)^{p-3}\}$.
In other words, we have to check the inequality
\[
\tag{*}
\big(\lceil\frac{p-1}{\sigma}\rceil + \ell\big)\cdot(k-1) + (2k'-1)(2 - \ell) \le p + 1
.
\]
As $k' \le k$, we may replace $k'$ by $k$. After that, one sees that the
worst case is the one with $\ell = 0$, so the remaining inequality is
\[
\tag{**}
\lceil\frac{p-1}{\sigma}\rceil\cdot(k-1) + 4k-3 \le p
.
\]
Estimating $\lceil\frac{p-1}{\sigma}\rceil \le \frac{p+\sigma-2}{\sigma}$
(and using the definition of $\sigma$) yields
$p \ge \frac{5k^{2}-4k-4}{k-2}$, i.e. $p \ge 29$ for $k=3$, $p \ge 30$ for
$k = 4$, $p\ge 33\frac{2}{3}$ for $k = 5$ and $p \ge 38$ for $k = 6$.
Thus it remains to check the cases $(k,p) = (4,29), (5,29), (5,31), (6,29),
(6,31), (6,37)$. One checks case by case that (**) holds in each of these
cases with exception of $k=6,p=29$. In this last case, we have
$m = 5$ and $k' \le \lceil\frac{p+1}{m+2-\ell}\rceil
= \lceil\frac{30}{7-\ell}\rceil$. If $\ell < 2$, this is equal to $5$ and
if $\ell = 2$, then $k'$ does not appear in (*), hence in (*) we may
replace $k'$ by $5$, $\ell$ by $0$ (which again is the worst case),
and we get $\lceil\frac{28}{15}\rceil\cdot 5 + 9\cdot 2 \le 30$, which is true.
\end{proof}

Theorem~\ref{thm:MaxMult} (1) now follows from
Lemmas~\ref{Lem:n-3:SmallMult}, \ref{Lem:k2}, and \ref{Lem:k36}.

\section{The three largest multiplicities of a zero-sum free set in
  $\Z_p^2$}
\label{sect:3MaxMult}

In this section we prove Theorem~\ref{thm:MaxMult} (4).

Let $A$ be a zero-sum free sequence, $m_1, m_2, m_3$ be the three largest
multiplicities, let $a_1, a_2,a_3$ be the elements with these
multiplicities, and let $\delta=2p-2-m_1-m_2-m_3$ be the number of
remaining elements ($0 \le \delta \le 3$).
We will prove our theorem by a series of restrictions on the possible
shape of $A$, each of which we state as separate lemmas.

In view of Theorem~\ref{thm:MaxMult} (1), we will always suppose
$\max(m_1, m_2, m_3)\leq p - 4$.
\begin{Lem}\label{Lem:atLeast35}
We can suppose that $p\geq 41$ and that $\min(m_1, m_2, m_3)\geq 13$.
\end{Lem}
\begin{proof}
The case $p \le 37$ is \refmc{3max} (which has been done by computer). Note that we
only have to choose 3 multiplicities and up to 6 elements in $\Z_p^2$,
hence, these computations are feasible even for rather large value of
$p$. The total computation time was 20 minutes.

The lower bound for $\min(m_1, m_2, m_3)$ follows from the fact that
the largest multiplicity is at most $p-4$, and the second largest is
less than $2p/3$.
\end{proof}

We will not in general assume that $m_1\geq m_2\geq m_3$,
but will restrict different conditions on these integers to exploit
symmetries more efficiently. Choose coordinates such that $a_1=(1,
0)$, $a_2=(0, 1)$. With 
respect to these coordinates we can represent $a_3$ as 
$(x,y)$; without further mentioning we fix this meaning of $x, y$.

\begin{Lem}
\label{Lem:No3y1} We have $y\neq 1$ (and, analogously, $x \neq 1$).
\end{Lem}
\begin{proof}
We first show that $(x, y)=(1, 1)$ is impossible. We try to form the
zero-sum $m_3(1, 1)+(p-m_3)(1, 0)+(p-m_3)(0, 1)$, which is possible,
unless $m_3+\min(m_1, m_2)<p$, that is, $\max(m_1, m_2)\geq
p-1-\delta \ge p-4$; since $(x, y)=(1, 1)$ we have still one symmetry
at our disposal and may suppose that $m_1\geq m_2$.
By part (1) of Theorem~\ref{thm:MaxMult}, we get $m_1=p-4$ and $\delta=3$.

Suppose first
that there is an element $a\in A$ different from $a_2, a_3$ satisfying
$\pi_2(a)=1$. We apply Lemma~\ref{Lem:2n-9Diff} to the equation
$\pi_2(a)=\pi_2((0, 1))$ and obtain a contradiction, unless
$|\pi_1(a)|\leq 2$. The same argument applied with $(1, 1)$ instead of
$(0, 1)$ yields $|\pi_1(a)-1|\leq 2$, thus, $a=(2, 1)$ or $a=(-1,
1)$. If there were such an element, we could form the zero-sum
\[
m_3(1, 1)+a+(p-m_3-1)(0, 1)+(p-m_3-\pi_1(a))(1, 0),
\]
note that the required multiplicity of $a_1$ poses no problem, since
\[
m_1 = p-4\geq p-m_3-\pi_1(a).
\]
We now apply Lemma~\ref{Lem:2n-9Diff} to the equation $\pi_2(3(0,1))
= \pi_2(3(1, 1))$, and obtain a contradiction, provided that
\[
-3\in\Sigma(\{1^{m_2+m_3-6}\} \cup (A\setminus\{a_1^{m_1}, a_2^{m_2},
a_3^{m_3}\}).
\]
Let $b_1, b_2, b_3$ be the three elements in $A$ different from $a_1,
a_2, a_3$. Since $m_2+m_3-6 = p-7 >p/2$, we get our contradiction
unless either $\pi_2(b_1)+\pi_2(b_2)+\pi_2(b_3)\leq 3$
(which is impossible), or one of the
three elements, say $b_1$, satisfies $\pi_2(b_1) =: -k \in\{-1, -2\}$.
Applying Lemma~\ref{Lem:replace} to $E := \{b_1,(1,1)^{k}\}$
yields a contradiction,
unless we have $b_1=(1, -k)$. However, even in these
cases we can apply part (1) of Lemma~\ref{Lem:replace},
thus, $A' = A\setminus\{b_1, (1,
1)^k\}\cup\{(1, 0)^{k+1}\}$ is zero-sum free. Since $m_3>3$, we find
that all elements in $A'$ satisfy $\pi_2(a)=0, 1$. However
$b_2$ and $b_3$ contradict this.

Hence, the assumption $(x, y)=(1, 1)$ leads to a
contradiction. Moreover, we can change the roles of $a_2$ and $a_3$
and find that $(x,y)=(-1,1)$ is also impossible.

Thus, $m_1=p-4$, and $|x|\geq 2$. From Lemma~\ref{Lem:2n-9Diff} we
immediately find $|x|\leq 2$, and exploiting the symmetry between
$a_2$ and $a_3$ we may assume that $x=2$. We now apply
Lemma~\ref{Lem:2n-9Diff} to the equation $\pi_2(2(0, 1))=\pi_2(2(x,
1))$, and obtain a contradiction, provided that
$-2\in\Sigma(\pi_2(A)\setminus\{1^4\})$. But $\pi_2(A)$ contains 1
with multiplicity $\geq p-5$, hence, we are done unless there is an
element in $A$ with $\pi_2(a)=-1$. But then we can replace $a$ and one
copy of $(2, 1)$ by at least three copies of $(1, 0)$, and therefore
obtain a zero-sum.
\end{proof}

\begin{Lem}
\label{Lem:no3y-1}
We have $y\neq -1$ (and, analogously, $x \neq -1$).
\end{Lem}
\begin{proof}
We now replace one copy of $(0, 1)$ and one copy of $(x, -1)$ by one
copy of $(x, 0)$, until we run out of elements of the form $(x, -1)$ or $(0,
1)$. In this way we obtain $\min(m_2, m_3)$ elements $(x, 0)$, hence,
for $A$ to be zero-sum free it is necessary that $\{1^{m_1},
x^{\min(m_2, m_3)}\}$ be zero-sum free. But $m_1+\min(m_2, m_3)\geq
p-1$, thus, this set is zero-sum free if and only if it is constant,
that is, $x=1$, and we are in the case covered by Lemma~\ref{Lem:No3y1}.
\end{proof}

\begin{Lem}
\label{Lem:threen-5}
We have $m_1<p-5$.
\end{Lem}
\begin{proof}
Suppose otherwise. Then $m_2, m_3\in[\frac{p-1}{2},
\frac{2p}{3}]$. Suppose that in the sequence $ny\bmod p$, $1\leq n\leq
m_3$, there are 5 elements in $[p-m_2, p]$. Since all these
elements have different value under $\pi_1$, one of them satisfies
$\pi_1(n(x, y))\not\in[1, 4]$, and this multiple can be combined with
certain copies of $(1, 0)$ and $(0, 1)$ to a zero-sum. If $m_3|y|>5p$,
we can choose integers $0<n_1<\dots<n_5\leq m_3$, such that $n_iy\bmod
p\in[p-y+1, p]$, and our claim follows. Hence, $|y|\leq 10<p/4$, that
is, the same argument yields $m_3|y|<3p$. We now repeat this argument
to obtain $|y|\leq 5<p/8$, which implies $m_3|y|\leq 2p$, which again
implies $|y|\leq 3$, which yields $m_3|y|<p$, which is only possible
if $m_3=\frac{p-1}{2}$, and $y=2$, that is, $a_3=2a_2+ya_1$. Choosing
$a_1, a_3$ as a basis we find that $a_2=2a_3+y'a_1$, that is,
$a_3=4a_3+y'' a_1$, which is impossible since $a_1$ and $a_3$ are
linearly independent, and $4\not\equiv 1\;(p)$. Hence, our claim follows.
\end{proof}

From now on we shall assume that $m_3$ is the least of the three
multiplicities. 
We continue to assume $a_1=(1, 0)$, $a_2=(0, 1)$ and $a_3 = (x,y)$,
and we choose $a_1$, $a_2$ in such a way that $x\geq y$. Note that the
upper bound $\max(m_1, m_2)\leq p-6$, immediately implies the lower
bounds $m_3>\frac{p}{3}+1$ and $\min(m_1,m_2)>\frac{p}{2}$. 

\begin{Lem}
\label{Lem:Nosmally}
We have the two inequalities
$y\geq\frac{m_1+m_2-p+2}{2}\geq\frac{p-4}{6}$; in particular, $y\geq 7$.
\end{Lem}
\begin{proof}
The second inequality just follows from $m_1+m_2\geq\frac{2(m_1+m_2+m_3)}
{3}\geq \frac{4p-10}{3}$.

For the first inequality, we distinguish the cases $|x| < y$ and $|x| \geq y$.
Suppose first $|x| < y$. By our general assumption $x \ge y$, we have
$p-x < y$. 
Let $k$ be the smallest integer such that $ky \ge p-m_2$.
Since $y \ge 2$ and $m_3 > \frac{p}{3}$, we have $k \le m_3$.
Assuming $y < \frac{m_1+m_2-p+2}{2}$, we want to show that $k\cdot
(p-x) \le m_1$ and $ky \le p$ to get a contradiction. By $p-x < y$, it
suffices to show that $ky \le m_1$. But $ky - m_1 \le p-m_2+y-m_1 <
\frac{p-m_1-m_2+2}{2} \le 0$ for $p \ge 16$.

Now suppose $|x| \ge y$. Set $k = \lceil\frac{p-m_2}{y}\rceil$ and
$\ell = \min(m_3, \lfloor \frac{p}{y}\rfloor)$.  We have $k \le \ell$
since $m_3\cdot y > (\frac{p}{3}+1)\cdot 2 \ge p-m_2$, so it makes
sense to consider the expressions $k\cdot(x, y)+(p-k y)\cdot(0, 1)$,
$(k+1)\cdot(x, y)+(p-(k+1) y)\cdot(0, 1)$, \dots, $\ell\cdot(x,
y)+(p-\ell y)\cdot(0, 1)$. By the choice of $k$ and $\ell$, each of
these expressions is contained in $\Sigma(\{a_2^{m_2}, a_3^{m_3}\})$,
and each one has second coordinate zero. Hence, we obtain an
arithmetic progression in $\langle(1,  0)\rangle$ of length $\ell - k+1$ with
difference $|x|$. This implies $(\ell - k)|x|\leq p-m_1-2$. We obtain
\[
|x|\bigg(\min\Big(m_3, \lfloor\frac{p}{y}\rfloor\Big)
-\lceil\frac{p-m_2}{y}\rceil\bigg)\leq p-m_1-2,
\]
which, by $|x|\ge y$ implies 
\[
\min(ym_3+m_2-p-y, m_2-2y) \leq p-m_1-2.
\]
If the first term in the minimum is smaller, we obtain (using $y \ge 2$)
$m_1+m_2+2m_3\leq pn$, which is impossible. Hence,
$y\geq\frac{m_1+m_2+2-p}{2}$.
\end{proof}

\begin{Lem}
\label{Lem:yBound}
Suppose that $m_3$ is the least of the three multiplicities, and that
$x\geq y$. Then $y>\frac{3}{10}n$.
\end{Lem}
\begin{proof}
For $p\geq 41$, we have $\frac{p-4}{6}>\frac{p}{7}$, hence, in
view of Lemma~\ref{Lem:Nosmally} we may assume that
$y>\frac{p}{7}$. Call an integer $k$ obstructing, if $k\leq m_3$, 
and $ky\bmod p\in[p-m_2, p]$. This definition is motivated by the fact
that if $k$ is obstructing, then
\[
\frac{x}{p}\in\bigcup\limits_{a=0}^{k-1}(\frac{a}{k},
\frac{a}{k}+\frac{p-m_1}{kp}),
\]
that is, we obtain obstructions on the possible values of $x$ (see
Figure~\ref{fig:obstruct}). For
different ranges of $y$, we obtain different obstructing integers, and
we will obtain a contradiction by showing that no possibility for $x$
remains.

\begin{figure}
\begingroup
\newlength\ole
\newlength\oy
\oy=0cm
% Change the value after "unit=" to change the width of the figure
\psset{unit=10cm,dimen=middle,linewidth=0.1mm,hatchwidth=0.1mm}
\SpecialCoor
\def\obstruct#1#2{%
  \ole=#2\psunit%
  \multirput(0,\oy)(\ole,0){#1}{\psline(0.5\ole,0)(\ole,0)}%
  \rput[r](-1mm,\oy){#1}%
  \advance\oy by -4mm%
}
\def\vertbar#1#2{%
  \psline(#1,1mm)(#1,\oy)%
  \rput(0,-1mm){\rput[t](#1,\oy){#2}}%
}

\begin{pspicture}(0,-25mm)(1,0)
\obstruct{2}{0.5}
\obstruct{3}{0.3333333}
\obstruct{4}{0.25}
\obstruct{5}{0.2}
\obstruct{6}{0.1666667}
\obstruct{10}{0.1}

\advance\oy by 3mm
\vertbar{0}{0}
\vertbar{1}{1}
\vertbar{0.1428571}{$\frac17$}
\vertbar{0.7}{$\frac7{10}$}
\vertbar{0.75}{$\frac3{4}$}
\vertbar{0.8}{$\frac45$}
\vertbar{0.8333333}{$\frac56$}
\vertbar{0.875}{$\frac78$}
\end{pspicture}
\endgroup
\caption{Obstructions on $\frac xn$ for $m_1 = \frac{n}{2}$ and different $k$.}
\label{fig:obstruct}
\end{figure}
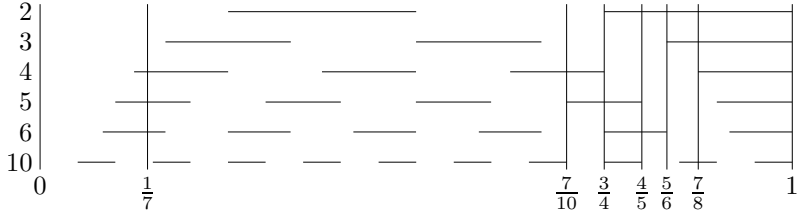

We first deal with the range $\frac{p}{7}<y\le\frac{p}{5}$.
Then 4, 5 and at least one of 3, 6 are
obstructing. Using the bound $m_1>p/2$ and $x \ge y > \frac{p}{7}$,
we obtain $\frac{x}{p}\in(\frac{4}{5},
\frac{7}{8})$, and that not both 3 and 6 can be obstructing. If
$y<\frac{p}{6}$, this implies that $m_2<\frac{4}{7}p$,
$\frac{x}{p}\in(\frac{5}{6},\frac{7}{8})$, and $m_1<\frac{2}{3}p$. Hence,
$2p-5<\frac{2}{3}p+\frac{8}{7}p=\frac{38}{21}n$, which is impossible
for $p\geq 41$. If
$y>\frac{p}{6}$, we obtain $\frac{x}{p}\in(\frac{4}{5}, \frac{5}{6})$,
and $m_1<\frac{3}{5}p$, hence $m_2 > \frac{4}{5}p - 5$.
For $p\geq 41$, we obtain
$m_2>\frac{2}{3}p$, which implies that 2 is obstructing, and gives a
contradiction.

Next, suppose that $\frac{p}{5}<y\le\frac{p}{4}$. If
$m_2\geq\frac{3}{5}p$, then $2, 3, 4$ are obstructing, and we
immediately obtain a contradiction. Otherwise, 3, 4 and 8 are
obstructing, and we obtain $\frac{x}{n}\in (\frac{3}{4},
\frac{2}{3}+\frac{p-m_1}{3p})$. Suppose that $y\leq\frac{2p}{9}$. Then
9 is obstructing, and we obtain that the intervals $(\frac{2}{3},
\frac{2}{3}+\frac{p-m_1}{3p})$ and $(\frac{7}{9},
\frac{7}{9}+\frac{p-m_1}{9p})$ overlap, which is only possible for
$m_1<\frac{2p}{3}$. But then
\[
2p-5\leq m_1+m_2+m_3\leq\frac{2p}{3}+\frac{6p}{5} = \frac{28}{15}p,
\]
which fails for $p\geq 41$. If $y>\frac{2}{9}$, then 2 is obstructing,
unless $m_2<\frac{5p}{9}$, but then
\[
2p-5\leq m_1+m_2+m_3\leq\frac{3p}{4}+\frac{10p}{9} = \frac{67}{36}p,
\]
which is also impossible.

If $\frac{p}{4}<y<\frac{3p}{10}$, then 2, 3, 6, and 10 are
obstructing, which implies 
$x\in(\frac{7}{10}, \frac{3}{4})$ and $m_1<\frac{3p}{5}$. If
$y\leq\frac{2p}{7}$, then 7 is obstructing, and we obtain
$m_1<\frac{4p}{7}$, which gives
\[
2p-5\leq m_1+m_2+m_3\leq\lfloor\frac{3p}{4}\rfloor+2\lfloor\frac{4p}{7}
\rfloor \leq \frac{53p}{28}.
\]
For $p>43$ this estimate gives a contradiction, while for $p=41, 43$
we compute explicitly the rounding errors and obtain a contradiction
as well. If $\frac{y}{p}\in(\frac{2}{7},\frac{3}{10}]$, then 5 is
obstructing, which yields a contradiction, unless
$m_2<\frac{4p}{7}$. But then $m_1+m_2+m_3\leq\frac{61}{35}$, which
is impossible.
\end{proof}

We can now finish the proof of Theorem~\ref{thm:MaxMult} (ii).

Consider the set $B=\{iy: 1\leq i\leq m_3, iy\bmod p\geq
p-m_2\}$. Then $|B|\leq p-m_1-1$, since for $i\neq j\in B$ we have
$ix\neq iy$. Hence, $C=\{iy: 1\leq i\leq m_3, iy\bmod p< p-m_2\}$
satisfies $|C|\geq m_1+m_3-p+1\geq p-4-m_2$, that is, there are at
most 3 values in the range $[1, p-m_2-1]$, which are not in
$C$.

Suppose first that $y<p-m_2$. For every element $c$ in $C$ with at
most one exception we have that $B$ contains $c+\nu y$ for all $\nu$
such that $c+\nu y\in[p-m_2, p]$, together with Lemma~\ref{Lem:yBound}
we deduce $|B|\geq m_2-11$. Hence,
$m_1+m_2\leq p+10$, thus $m_1+m_2+m_3\leq\frac{3p}{2}+15$, which is
impossible for $p\geq 41$. 

If $y\geq p-m_2$, then 1 is obstructing, which implies $x\in[1,
p/2]$. By our assumption we have $y\leq x$, hence $2y<p$, and we
obtain a zero-sum, unless $2x<p-m_1$. But then $y\leq x<p/4$, which
contradicts Lemma~\ref{Lem:yBound}. Hence, Theorem~\ref{thm:MaxMult}
(ii) is proven.

\section{Asymptotic estimates}
\label{sect:MinMult}
\subsection{Lower bounds for the largest multiplicities}
We first establish the following, which is a strengthening of the
bound for $m_1$ implied by Theorem~\ref{thm:MinMult}.
\begin{Theo}
\label{thm:pMult}
For every $\epsilon>0$ there exists some $\delta>0$, such that for
every sufficiently large prime number $p$ and every multiset
$A\subseteq\Z_p^2$ such that no element of $A$ has multiplicity
$\geq\delta p$, the following holds true.
\begin{enumerate}
\item If $|A|>(1+\epsilon)p$, then $A$ contains a zero-sum of length $\leq p$.
\item If $|A|>(2+\epsilon)p$, then $A$ contains a zero-sum of length $p$.
\end{enumerate}
\end{Theo}
We will need the following lemma.
\begin{Lem}
\label{Lem:AD}
There exists an absolute constant $W$, such that the following holds
true: If $p$ is a sufficiently large prime, and
$A\subseteq\Z_p^2$ is a set with $|A|\geq p/4$, and if for each affine
line $L$ we have $|A\cap L|\leq \frac{|A|}{W}$, then there exists some $n$
such that $|\Sigma^n(A)|\geq p^2/2$.
\end{Lem}
\begin{proof}
The proof follows closely the lines of the induction step in
Section~2.3 of \cite{AD}. In fact, the only
changes necessary affect the choice of $s$ in \cite[equation (7)]{AD},
which we have to choose $\leq p/24$ to ensure that after using $3s$
elements the remaining set $A'$ still has the property that for each
affine line $A'$ we have $|A'\cap L|\leq\frac{2|A'|}{W}$.
\end{proof}

\begin{proof}[Proof of Theorem~\ref{thm:pMult}.]
Define $W$ as in Lemma~\ref{Lem:AD}. We distinguish two cases,
depending on whether there exists an affine line containing at least
$\frac{p}{W}$ elements of $A$ or not. Suppose first, that no such line
exists. Choose subsets $A_1, A_2\subseteq A$ with
$|A_i|=\lceil\frac{p}{4}\rceil$. Then both $A_1, A_2$ satisfy the conditions of
Lemma~\ref{Lem:AD}, hence, there exist some $n_1,n_2\leq p/2$ such that
$|\Sigma^{n_i}(A_i)|\geq p^2/2$. For statement (1) this is already sufficient,
since $\Sigma^{n_1}(A_1)\cap(-\Sigma^{n_2}(A_2))\neq\emptyset$,
and we obtain a
zero-sum of length $n_1+n_2\leq p$. Note that $n_1, n_2$ cannot be
zero, that is, this zero-sum is in fact non-trivial. For statement (2)
we choose $p-n_1-n_2$ arbitrary 
elements in $A\setminus(A_1\cup A_2)$, add them up to obtain an
element $s$, and use the fact that
$\Sigma^{n_1}(A_1)\cap(-s-\Sigma^{n_2}(A_2))\neq\emptyset$
to find a zero-sum using
$n_1$ elements in $A_1$, $n_2$ in $A_2$, and $p-n_1-n_2$ in
$A\setminus(A_1\cup A_2)$. Hence, in this case our claim follows.

Now suppose that there exists a line $L$ with $|A\cap L|\geq \frac{p}{W}$.
For statement (1), if this line passes through $0$, we obtain a
zero-sum using Lemma~\ref{Lem:GGmult}, provided that
$\delta<\frac{1}{40W^2}$. For statement (2) we can add
a vector to all elements in $A$ without changing the statement, hence,
in both cases we may assume that $L=\{(1, t):t\in\Z_p\}$. If
$\delta<\frac{\epsilon}{100W^2}$, then from
$A\cap L$ we can choose $\lfloor\frac{\epsilon^2 p}{400W}\rfloor$ sets $B_i$
containing $100\epsilon^{-1}W$ 
different elements each, and set $B=\bigcup B_i$;
note that $|B| < p\epsilon/4$. From Lemma~\ref{Lem:ksubsets} it
follows that $|\Sigma_k(B_i)|\geq 2500\epsilon^{-2}W^2$, where
$k=\lfloor|B_i|/2\rfloor$. Hence, putting $N=k\lfloor\frac{\epsilon^2
  p}{400W}\rfloor$ it follows from the Cauchy-Davenport theorem that
$\Sigma_N(B)$ contains the whole line $\{(N, t): t\in\Z_p\}$.
Hence, our claim follows if we can show for statement (1) that every element
of $\Z_p$ can be written as a subset sum of $\pi_1(A\setminus B)$, and for
statement (2) that every element in $\Z_p$ can be written as a subset sum of
$\pi_1(A\setminus B)$ of length $p-N$. Suppose that this is not the case.
For statement (1) this implies that
$\pi_1(A\setminus B)$ contains less than $p$ non-zero elements.
However, in this case 
$\pi_1(A\setminus B)$ contains 0 with multiplicity at least
$\frac{3\epsilon}{4} p$, so we may
apply Lemma~\ref{Lem:GGmult} once more to
obtain a zero-sum. For statement (2) note that $N\sim\epsilon
p/4$. Hence, we obtain a zero-sum, unless there is some element $a\in
\Z_p$, such that $A$ contains at 
least $(1+\epsilon/2)p$ elements mapping to $a$ under $\pi_1$. But then we
find a zero-sum of length $p$ within this set in the same way as 
for statement (1).
\end{proof}
We now turn to the proof of Theorem~\ref{thm:MinMult}. Assume that
$(1, 0)$ is the point with the highest multiplicity $m_1$ in $A$. If
$m_1<(1-\epsilon)p-2$, set $A'=A\setminus\{(1, 0)^{m_1}\}$. Then by
Theorem~\ref{thm:pMult} we see that $A'$ contains a zero-sum, unless
the largest multiplicity of $A'$ is at least $\delta p$ for some
$\delta$ depending on $\epsilon$. Hence, it suffices to consider the
case $m_1>0.9 p$.

Choosing $W$ as in Lemma~\ref{Lem:AD}, we find that $A'$ contains a
zero-sum, unless there is a line $L$ with $|A'\cap
L|>\frac{p}{W}$. Again as in the proof of Theorem~\ref{thm:MinMult} we
see that for $\delta$ sufficiently small we can find a set $B\subseteq
A\cap L$ with $|B|<0.1 p$ such that $\Sigma(B)$ contains some line
$L'=\{(a, b)+(x, y)t:t\in\Z_p\}$. Suppose first that $(x, y)$
is not collinear to $(1,0)$. Then $\langle(x, y)\rangle$ contains at
most $\delta$ elements of $A$, hence  in $A\setminus B$ we find
$p-1$ elements not collinear to $(x, y)$. Thus we can find an element
$s\in\Sigma(A\setminus B)$ with $-s\in L'$; together with some elements
in $B$, this yields a zero-sum.

Now we suppose that $L'$ is
parallel to $\langle(1, 0)\rangle$. We obtain a zero-sum
if $\Sigma(\pi_2(A\setminus B))=\Z_p$. Since $A\setminus B$
contains at least $2p-2-0.1p-m_1\geq 0.9 p$ elements, this is
certainly the case unless there is some $a\in\Z_p$, such that
$|\pi_2^{-1}(a)|>0.8p$. Thus we may assume without loss that $A$ contains
at least $0.8p$ elements $a$ with $\pi_2(a)=1$ (and $(1, 0)$ with
multiplicity $>0.9p$).

For $\delta$ sufficiently small we can easily find
$p/20$ pairs $a_1, a_2$ in $A$, such that $\pi_2(a_1)=\pi_2(a_2)=1$,
and $|\pi_1(a_1)-\pi_1(a_2)|>10$. If there is a pair with
$|\pi_1(a_1)-\pi_1(a_2)|>0.1 p$, we are immediately done
by Lemma~\ref{Lem:2n-9Diff}. Otherwise we take
$N=\lfloor\frac{p-m_1-1}{2}\rfloor \le p/20$ such pairs. Since there are
$2p-2-m_1-2\lfloor\frac{p-m_1-1}{2}\rfloor\geq p-1$ elements in $A$
which are neither in one of the pairs chosen nor equal to $(1, 0)$,
there is an element $s$ with $\pi_2(s)=-N$, which can be represented
using elements not in one of the chosen pairs nor equal to
$(1,0)$. Choosing one element of each pair and adding them to $s$
yields an element of $\langle(1, 0)\rangle$; by using different choices,
we obtain a sequence of $N+1$ elements
$(x_0, 0), \ldots (x_N,0) \in \Sigma(A\setminus\{(1, 0)^{m_1}\})$
with $10<x_{i+1}-x_i<0.1 p$. This yields a zero-sum unless
$0 < x_0 < x_N < p - m_1$, i.e. we get
$10N < p - m_1$. But $10N \ge  5(p-m_1)-10$,
which contradicts $p - m_1 \ge 3$.

\medskip

If the reader has the impression that our dealing with constants in
the proof of Theorem~\ref{thm:MinMult} is quite wasteful, she is
certainly right. However, the real loss occurs in the use of
Lemma~\ref{Lem:AD}, and we did not try to improve a constant which
will still be too small to be of much use.

\subsection{Upper bounds for the largest multiplicity}
In this section we prove Theorem~\ref{thm:MaxMult}(\refmm{asymp}). Let $p$ be a
prime number, $A\subseteq\Z_p^2$ be a zero-sum free set with
$|A|=2p-2$, and maximal multiplicities $m_1 \ge m_2$. We
may assume that 
the elements with maximal multiplicity are $(1, 0)$ and $(0, 1)$,
and that $A$ contains no other element of the form $(x,0)$ or $(0,y)$.
Set $\delta=p-m_1$; in several places, we will suppose that
$\delta/p$ is sufficiently small (but independently of $p$).
We will moreover use the following definitions:
$\mu$ is the maximal multiplicity of $\pi_2(A\setminus\{(1, 0)^{m_1}\})$, 
and $k=\lceil\frac{p}{m_2}\rceil$ is the ``number of times one would
need the elements $(0,1)^{m_2}$ to fill an entire $\Z_p$''.

We do already have a lower and an upper bound for $m_2$:
by \refmm{2max}, we may suppose $m_2 < 2p/3$. On the other hand,
for $\delta/p$ sufficiently small, Theorem~\ref{thm:MinMult} yields:

\begin{Lem}\label{Lem:m2lowerBound}
We have $m_2 > 8\delta$, and in particular $k \le \frac{p}{4\delta}$.
\end{Lem}

We will now first get precise statements describing the
rows $A \cap \pi_2^{-1}(y)$; the important result here is
Lemma~\ref{Lem:muSmall}, which bounds the number of elements
of each row.
Then we use the method of Lemma~\ref{Lem:CompAppl} to finish the proof.

We proceed by induction in the following way. Let $A'$ be another set
with cardinality $2p-2$ and maximal multiplicities $p - 3 \ge m'_1 \ge m'_2$.
We suppose that the
claim is true if $m'_1 \ge m_1$, $m'_2 \ge m_2$ and
$(m_1,m_2) \ne (m'_1, m'_2)$. Moreover,
for $B \subset \Z_p^2$ consider the sum
\[
S(B):=\sum_{(x,y)\in B}\ii(x)^2 + \ii(y)^2.
\]
We also suppose that the claim is true for $A'$ if $m'_1 = m_1$ and $m'_2 = m_2$
and $S(A') > S(A)$.

Using this induction, we show:
\begin{Lem}\label{Lem:horizPairs}
Suppose $(x,y), (x',y) \in A$ with $y \ge 2$. Then $x - x' \in \{-1,0,1\}$.
\end{Lem}
\begin{proof}
Suppose otherwise. After possibly exchanging $x$ and $x'$, we may
suppose $\ii(x' - x) \le p - \delta + 1$.
Then $\Sigma(\{(1,0)^{p-\delta}, (x,y), (x',y)\})$ contains
the whole interval $(x,y) + \{0, 1, \dots, \ii(x' - x) + p - \delta\}\cdot(1,0)$.
In particular, if we replace $(x,y)$ and $(x',y)$ by $(x+k,y)$ and $(x'-k,y)$
for some $0 \le k \le \ii(x' - x)$,
then we get a new set $A'$ satisfying $\Sigma A' \subset \Sigma A$. Thus
it suffices to prove that $A'$ contains a zero-sum.
If $\ii(x') > \ii(x)$, then choose $k = 1$.
As $\ii(x+1)^2 + \ii(x'-1)^2 > \ii(x)^2 + \ii(x')^2$, the set $A'$ contains a
zero-sum by induction.
If $\ii(x) < \ii(x')$, then choose $k = \ii(x')$. Then $A'$ contains $(0,y)$,
which is impossible.
\end{proof}

\begin{Lem}\label{Lem:muSmall}
We have $\mu\leq m_2+\delta-2$.
\end{Lem}
\begin{proof}
Let $B := \pi_2(A \setminus \{(1,0)^{p-\delta}\})$, and let $y$ be an
element of maximal multiplicity of $B$; we assume that this multiplicity
is at least $m_2 + \delta - 1$.
By Lemma~\ref{Lem:m2lowerBound}, $m_2 \ge \delta$, so we may
set $B' := B \setminus \{ y^{2\delta-2}\}$.
We claim that
if $\Sigma(B')$ contains $-(\delta - 1)y$, then $A$ contains a zero-sum.

Choose an element $a\in\sigma(A)$ with $\pi_2(a)=-(\delta - 1)y$, and
form $\delta-1$ 
pairs $(x_i,y), (x'_i,y) \in A$ with $x_i \ne x'_i$, that is,
$x_i=x_i'\pm 1$. We have
$|\Sigma\{x'_i-x_i\mid 1 \le i \le \delta-1\}|=\delta$, thus by taking
$a$ and one element of each pair, we get
$\delta$ different sums in $\langle (1,0) \rangle$. Together with
$(1,0)^{p-\delta}$, one of them yields a zero-sum.
This proves the claim, hence it remains to show
that $\Sigma(B')$ contains $-(\delta - 1)y$.

As $|B'| = p-\delta$ we have $\Sigma(B') = \Z_p$ unless $B'$ contains an
element $y'$
with multiplicity at least $p - 2\delta + 2$.
As this is more than $|B|/2$
and $y$ was chosen maximal, this implies $y' = y$;
thus $B$ contains $y$ with multiplicity at least $p$.

If $y \ne 1$, then there are only $\delta - 2$ elements left in $A$
which might be equal to $(0,1)$. This contradicts Lemma~\ref{Lem:m2lowerBound}.
so we have $y = 1$, and our task simplifies to proving that $-(\delta
- 1) \in \Sigma(B')$. 
If $B = \{1^{p-2+\delta}\}$, then
$A$ contains a zero-sum by Lemma~\ref{Lem:coset}, so we may suppose
$\sum_{b \in B'} \ii(b) \ge p-\delta + 1$. If $B'$ does not contain
any element in $[p - \delta + 2, p-1]$, then this together with the
high multiplicity of $1$ in $B'$ already implies
$-(\delta - 1) \in \Sigma(B')$, which is what we had to show.

So now let $d\in A$ be an element with
$\pi_2(d)\ge p-\delta+2$. Consider the set $S$ of all elements reachable from
$d$ by adding $p-\ii(\pi_2(d))$ elements $a \in A$ each satisfying $\pi_2(a)=1$.
By Lemma~\ref{Lem:replace}, any $s \in S$ satisfies $1 \le \pi_1(s) \le
p-\ii(\pi_2(d))$,
which is only possible if the set of elements in $A$
with $\pi_2(a)=1$ takes the form $\{(0, 1)^{m_2}, (\pm 1,
1)^{\mu-m_2}\}$. As $\mu\geq p$, we may form the sum
$m_2 \cdot (0,1) + (p-m_2)\cdot(\pm 1,1) = (\mp m_2, 0)$.
Together with copies of $(1,0)$ this yields a zero-sum
as $\delta \le m_2 \le m_1$.
\end{proof}
Recall that we defined $k=\lceil\frac{p}{m_2}\rceil$ and that we already
proved $k \le \frac{p}{4\delta}$.
\begin{Lem}
$A$ contains a zero-sum.
\end{Lem}
\begin{proof}
We will apply Lemma~\ref{Lem:CompAppl}.
We will decompose $A\setminus\{(1, 0)^{m_1}, (0, 1)^{m_2}\}$ into two subset $U$ and $V$
with $|V| = (2k-1)(\delta-1)$; this implies that $V$ satisfies the condition of the lemma.
We claim that by choosing $U$ appropriately, we may ensure that
the maximal multiplicity of $U' := \pi_2(U \cup \{(0,1)^{m_2}\})$
is at most $m_2$. Indeed, using $\mu \le m_2 + \delta - 2$, there are at most
$(\delta-2) \cdot \frac{2p-2 - (p-\delta)}{m_2 + \delta - 2} \le (\delta-2)\cdot\frac{p}{m_2}
\le (\delta-2)k \le (2k-1)(\delta-1)$ elements which we are
forced to include in $V$.

We have $|U'| = p - 2k\delta + 2k + 2\delta - 3$, and we want to show
that $\Sigma(U')=\Z_p$. For any fixed constant $c_0$ (say, $c_0 = 10$), 
$k\leq c_0$ implies $|U'|>5p/6$ if we choose $\delta/p$ small enough.
Using $m_2 < 2p/3$, we see that $\Sigma(U')=\Z_p$.

Now suppose $k \geq 11$, i.e.\ $m_2 < \frac{p}{10}$.
Then we can
partition $U'$ into subsets consisting of 10 different elements each,
leaving at most 9 elements unused. Each of these subsets has a sumset of
cardinality at least $29$ by Lemma~\ref{Lem:OlsonSize},
and the total number of sets is $\lfloor \frac{|U'|}{10} \rfloor$.
Now $k \le \frac{p}{4\delta}$ implies $|U'|>p/2$, so using
Cauchy-Davenport, we obtain $\Sigma(U')=\Z_p$, provided
that
\[
\left\lfloor\frac{p}{20}\right\rfloor 29\geq p-1
\]
which is certainly true for $p>100$.
\end{proof}

\section{Algorithms to check $B(n)$}
\label{sect:algo}

We now describe the algorithm used to prove Theorem~\ref{thm:MaxComp}.
All statements except (\ref{mc:2max}) use the same algorithm, described in
the first subsection. Statement (\ref{mc:2max}) is different:
it concerns arbitrarily large primes, and a priori the problem is
not finite. We will describe our approach in the second subsection.

\subsection{Algorithm for $n$ fixed}

In this subsection we work in $\Z_n$ for $n$ not necessarily prime
(because of the cases $8$, $9$ and $10$).

We will need the following lemma:

\begin{Lem}\label{Lem:algoHelp}
Suppose $A \subset \Z_n^2$ contains
$\{(1,0)^m, (x_1,y)^k, (x_2,y)^k\}$ where
$|x_1 - x_2| \le m + 1$, $p - k\cdot|x_1 - x_2| \le m+1$
and $|A| \ge 2k + m + n - 1$. Then $A$ contains a zero-sum.
\end{Lem}
\begin{proof}
By the two prerequisites concerning $|x_1 - x_2|$, any interval
$[a, a+m] \subset \Z_n$ contains an element of the form
$\ell \cdot x_1 + (k-\ell)\cdot x_2$ with $0 \le \ell \le k$; thus
$\Sigma\{(1,0)^m, (x_1,y)^k, (x_2,y)^k\}$ contains the whole
coset $\Z_n \times \{ky\}$.
By the last prerequisite, we can find a subset
of $A \setminus \{(1,0)^m, (x_1,y)^k, (x_2,y)^k\}$
whose sum $s$ satisfies $\pi_2(s) = -ky$; this yields a zero-sum.
\end{proof}

The algorithm to check property B in principle just tries every possible
multiset $A \subset \Z_{n}^{2}$ consisting of $2n - 2$ elements and
having maximal multiplicity at most $n - 3$
(and which, for statement $(\ref{mc:3max})$, satisfies the additional
condition concerning the three maximal multiplicities); however, we
need some good methods to reduce the computation time.
There are several such methods which only work when $p$ is prime;
as the non-prime cases we are interested in are relatively small,
this is not such a problem.

Let us first suppose that $n$ is prime.
Then we may fix that the two elements with maximal multiplicities
$m_1 \ge m_2$ are
$a_{1} = (1,0)$ and $a_{2} = (0,1)$.
The algorithm has two outer loops to try all possible 
values $m_{1}$ and $m_{2}$ and then
recursively adds other elements with smaller multiplicities.
This is done in the order of decreasing multiplicity, as elements with
higher multiplicity tend to yield contradictions more quickly.

During the computation, we always keep an up-do-date copy of the sumset
$\Sigma A$. Moreover, for each element $z \in \Z^2_n$ which is not yet
contained in
$A$, we store an upper bound for the multiplicity
$z$ can have in $A$. These bounds are updated each time a new element $a$
is added to $A$:
\begin{itemize}
\item
No negative of any existing subset sum may be added anymore.
(The corresponding upper bounds are set to zero).
\item
No other element of the subgroup $\langle a \rangle$ may be added
anymore by Lemma~\ref{Lem:Gaop}.
\item
Applying Lemma~\ref{Lem:algoHelp} with $(x_1,y) = a$ yields upper bounds
for the multiplicity of several elements of the form $(x_2,y)$.
\end{itemize}
Using these upper bounds, after each addition of an element
we try to estimate whether there is still enough room for all remaining elements
to be added (and stop if this is not the case).
If we are adding elements with multiplicity $k$ right now,
and there are $\ell$ cyclic subgroups left which are not yet completely
forbidden for new elements, then we have space left for $k\ell$ elements
at most (again using Lemma~\ref{Lem:Gaop}).

If $n$ is not prime, we can not apply Lemma~\ref{Lem:Gaop}. Moreover,
we do not know whether the two elements with maximal multiplicities $a_1$,
$a_2$ generate the group. However, we may always apply a group automorphism
such that $\pi_1(a_1) \mid n$ and $\pi_2(a_1) = 0$; moreover, if
$\pi_2(a_2) \ne 0$ we may apply a second
group automorphism, fixing $a_1$ and such that
$\pi_2(a_2) \mid n$ and $\pi_1(a_2) \in [0,\pi_2(a_2)-1]$.
Thus if $n$ is not prime, the algorithm has additional outer loops
iterating through all $a_1, a_2$ which
are possible after the application of such automorphisms.

Verifying \refmc{3max} took 5 minutes.
For (\ref{mc:all}), the total computation time
(distributed on several computers) was 2 hours for all cases up to $n = 17$,
31 hours for $n = 19$, and
196 days kindly provided by the Rechenzentrum Universit\"at Freiburg
for $n = 23$. The moreover-part ($n=8,9,10$) took 4 minutes.

\subsection{Algorithm for two large multiplicities and $n$ arbitrary}

We now turn to statement (\ref{mc:2max}) of Theorem~\ref{thm:MaxComp}.
We use notation from Section~\ref{sect:2MaxMult}: let $A \subset \Z_p^2$
be zero-sum free and of cardinality $2p-2$, let
$m_2 \le m_1 \le p-3$ be the two maximal multiplicities, and set $k_i := p - m_i$.
As we assume $m_2 \ge 2p/3$, we may apply lemmas from
Section~\ref{sect:2MaxMult};
the main ingredient to turn the problem into a finite one is our knowledge
about $A$ described in Figure~\ref{fig:BCD} (on page \pageref{fig:BCD}).

Fix $k_{1}$ and $k_{2}$ (the computer iterates through all pairs $k_1,k_2$
with $k_1+k_2 \le 14$), and
define $L$ to be the area marked with $B$, $C$ and $D$ in the figure, but
turned into a subset of $\Z^{2}$ in such a way that $L$ is independent of $p$:
\[
\begin{aligned}
L := \,&[1, k_1- 2] \times [1, k_2-2] \,\cup\\
&
\big([-k_1+ 2,-1] \times [1, k_2-2] \cap \{(x,y)  \mid x+y \le 1 \}
\big)\,\cup\\
&
\big([1, k_1-2] \times [-k_2+2,-1] \cap \{(x,y)  \mid x+y \le 1 \}
\big)
\end{aligned}
\]

The computer recursively considers every subset $A' \subset L$ of
cardinality $\ell := k_1+k_2-2$. To know whether $A = A' \cup \{(1,0)^{m_{1}} \cup
(0,1)^{m_{2}}\}$ has a zero-sum in $\Z_{p}^{2}$,
it has to check whether
$A'$ has a subset with sum $s$ such that $\pi_1(s) \in [k_1, p] \mod p$
and $\pi_2(s) \in [k_2, p] \mod p$. So each $s \in \Sigma(A')$ yields a condition on $p$,
and the question is whether all these conditions together exclude all
$p$. For $p$ sufficiently large, whether or not such a condition holds true does not depend
on $p$ anymore. Indeed, for any $s \in \Sigma(A')$ we have
$|\pi_i(s)| \le \ell\cdot(k_i-2)$, so
if $p - k_i \ge \ell\cdot(k_i-2)$, then
$\pi_i(s) \notin [k_i, p] \mod p$ implies $\pi_i(s) \in [1,k_i-1]$;
this is independent of $p$. This means that it suffices to consider values for $p$
only up to $\max(k_1 + \ell\cdot(k_1-2), k_2 + \ell\cdot(k_2-2))$;
in this way, the problem becomes finite.
(However, the case $p = \max(k_1 + \ell\cdot(k_1-2), k_2 + \ell\cdot(k_2-2))$
has to be checked even if this is not prime.)

Some efficiency improvements which we apply:
\begin{itemize}
\item
While we build $A'$ recursively, we maintain a list of possible values
for $p$. When we add a new element $a$ to $A'$, we go through all subset sums $s$
of $A'$ containing $a$ and update this list
accordingly. As soon as it is empty, we stop considering that case.
\item
We add elements $a$ to $A'$ in the order of decreasing
$|\pi_1(a)| + |\pi_2(a)|$. Elements where this value is high are likely
to yield a contradiction quickly, so we prefer to eliminate them
right at the beginning (instead of having to try to add each of them
to every almost completed set $A'$ which we get during our computation).
\end{itemize}

The running time was 10 seconds.

\end{document}